\documentclass[11pt,twoside,a4paper]{amsart}
\usepackage{amssymb}
\usepackage{amsmath,amscd}
\usepackage[initials,nobysame,shortalphabetic]{amsrefs}
\usepackage[all,cmtip]{xy}
\usepackage[hidelinks]{hyperref}
\usepackage{hyphenat}

\usepackage[utf8]{inputenc}

\usepackage{geometry}
\def\dbar{\bar\partial}
\def\R{{\mathbb R}}
\def\C{{\mathbb C}}

\def\Cn{\C^n}

\def\PM{{\mathcal{PM}}}

\def\Dom{{\rm Dom\,  }}
\def\Z{{\mathbb Z}}

\def\Ok{{\mathcal O}}

\newcommand{\Com}[1]{}

\DeclareMathOperator{\supp}{supp}
\DeclareMathOperator{\esssupp}{esssupp}

\DeclareMathOperator{\dist}{dist}
\DeclareMathOperator{\Reg}{Reg}

\def\be{\begin{equation}}
\def\ee{\end{equation}}

\newtheorem{thm}{Theorem}[section]
\newtheorem{lma}[thm]{Lemma}
\newtheorem{cor}[thm]{Corollary}
\newtheorem{prop}[thm]{Proposition}

\theoremstyle{definition}

\theoremstyle{remark}

\newtheorem{preremark}{Remark}
\newtheorem{preex}{Example}

\newenvironment{ex}{\begin{preex}}{\end{preex}}

\numberwithin{equation}{section}

\begin{document}

\title{Koppelman formulas on the $A_1$-singularity}

\date{\today}

\author{Richard L\"ark\"ang}

\address{Richard L\"ark\"ang, Department of Mathematics, University of Wuppertal, Gau{\ss}str. 20, 42119 Wuppertal, Germany, and Department of
  Mathematics, Chalmers University of Technology and the University of Gothenburg, 412 96 G\"oteborg, Sweden.}
\email{larkang@chalmers.se}

\author{Jean Ruppenthal}

\address{J. Ruppenthal, Department of Mathematics, University of Wuppertal, Gau{\ss}str. 20, 42119 Wuppertal, Germany}

\email{ruppenthal@uni-wuppertal.de}



\begin{abstract}
In the present paper, we study the regularity of the Andersson--Samuelsson Koppelman integral operator on the $A_1$-singularity.
Particularly, we prove $L^p$- and $C^0$-estimates. As applications, we obtain $L^p$-homotopy formulas
for the $\dbar$-equation on the $A_1$-singularity, and we prove that the $\mathcal{A}$-forms introduced by Andersson--Samuelsson
are continuous on the $A_1$-singularity.
\end{abstract}

\maketitle

\section{Introduction}

In this article, we study the local $\dbar$-equation on singular varieties.
In $\Cn$, it is classical that the $\dbar$-equation $\dbar f = g$, where $g$ is a $\dbar$-closed $(0,q)$-form,
can be solved locally for example if $g$ is in $C^\infty$, $L^p$ or $g$ is a current, where the solution
$f$ is of the same class (or in certain cases, also with improved regularity).
To prove the existence of solutions which are smooth forms or currents, or to obtain
$L^p$-estimates for smooth solutions, one can use Koppelman formulas, see for example, \cite{Range}*{LiMi}.

On singular varieties, it is no longer necessarily the case that the $\dbar$-equation is locally
solvable over these classes of forms, as for example on the variety $\{ z_1^4 + z_2^5 + z_2^4 z_1 = 0 \}$,
there exist smooth $\dbar$-closed forms which do not have smooth $\dbar$-potentials,
see \cite{RuDipl}*{Beispiel~1.3.4}.

Solvability of the $\dbar$-equation on singular varieties has been
studied in various articles in recent years, for example describing
in certain senses explicitly the obstructions to solving the $\dbar$-equation
in $L^2$, see \cites{FOV,OV2,RuDuke}.
Among these and other results, one can find examples when the $\dbar$-equation
is not always locally solvable in $L^p$, for example when $p = 1$ or $p = 2$.

On the other hand, in \cite{AS2}, Andersson and Samuelsson define on an arbitrary
pure dimensional singular variety $X$ sheaves $\mathcal{A}^X_q$ of $(0,q)$-currents,
such that the $\dbar$-equation is solvable in $\mathcal{A}$, and the solution is
given by Koppelman formulas, i.e., there exists operators $\mathcal{K} : \mathcal{A}_{q} \to \mathcal{A}_{q-1}$
such that
if $\varphi \in \mathcal{A}$, then
\begin{equation} \label{eqkoppel}
    \varphi(z) = \dbar \mathcal{K}\varphi(z) + \mathcal{K} (\dbar \varphi)(z),
\end{equation}
where the operators $\mathcal{K}$ are given as
\begin{equation*}
    \mathcal{K}\varphi(z) = \int K(\zeta,z) \wedge \varphi(\zeta),
\end{equation*}
for some integral kernels $K(\zeta,z)$.
The sheaf $\mathcal{A}_q$ coincides with the sheaf of smooth $(0,q)$-forms on $X^*$, where
$X^*$ is the regular part of $X$.
For the cases when the $\dbar$-equation is not solvable for smooth forms, 
the $\mathcal{A}$-sheaves must necessarily have singularities along $X_{\rm sing}$,
but from the definition of the $\mathcal{A}$-sheaves, it is not very apparent
how the singularities of the $\mathcal{A}$-sheaves are in general.
In order to take better advantage of the results in \cite{AS2}, one would
like to know more precisely how the $\mathcal{A}$ singularities of the
$\mathcal{A}$-sheaves are.
In particular, it would be interesting to know whether for certain varieties,
the $\mathcal{A}$-sheaves are in fact smooth, or, say, $C^k$ also over $X_{\rm sing}$.

In this article, we will consider solvability of the $\dbar$-equation on the so-called
$A_1$-singularity which is defined by
\begin{equation*}
    \{ \zeta_1 \zeta_2 - \zeta_3^2 = 0 \} \subseteq \C^3.
\end{equation*}
Our main method of study will be to study mapping properties of the Koppelman formulas
for the $\dbar$-equation from \cite{AS2}.

The motivation for us to do this is two-fold: First of all, as in the smooth case,
using integral formulas for studying the $\dbar$-equation has the advantage that it
can be used to studying the $\dbar$-equation over various function spaces,
like forms which are $C^k$, $C^\infty$, H\"older, $L^p$ or currents. 
Various results about solvability of the $\dbar$-equation on the $A_1$-singularity are
contained in earlier articles, as will be elaborated on below, and hence,
one wouldn't expect to obtain so much new results for this variety.
But thanks to the simplicity of the $A_1$-singularity, it serves as a good testing
ground for the method.
However, since the Koppelman formulas are defined for arbitrary pure dimensional varieties,
there is hope to extend the methods used here to more general varieties, and thus obtain
new results on such varieties about the solvability of the $\dbar$-operator over various functions spaces.
In particular, it seems likely that with some elaborations of the methods
here, that one should be able to extend the results here also to all rational double points.
The underlying idea and hope is that integral formulas -- as on manifolds -- will open the
door to further explorations. Let us just mention e.g. that 
it is usually easy to show that an integral operator is compact.
So, one gets compact solution operators for the $\dbar$-equation.
From that one can also deduce compactness of the $\dbar$-Neumann operator.

A second motivation is the following: the $\mathcal{A}$-sheaves in \cite{AS2}
are defined by starting with smooth forms, applying Koppelman formulas, multiplying
with smooth forms, applying Koppelman formulas, and iterating this procedure a
finite number of times. In the particular example of the $A_1$-singularity,
we obtain for example the new result that the $\mathcal{A}$-sheaves
are contained in the sheaves of forms with continuous coefficients,
see Corollary~\ref{cor:asheaves} below.

We will now describe the main results in this article:
From now on, we let $X$ be the variety given by
\begin{equation*}
    X = \{ \zeta \in B_1(0) \mid \zeta_1 \zeta_2 - \zeta_3^2 = 0 \} \subseteq \C^3,
\end{equation*}
where $B_r(0)$ is the ball of radius $r$ in $\C^3$. In addition, we let
\begin{equation*}
    X' = \{ \zeta \in B_{1+\epsilon}(0) \mid \zeta_1 \zeta_2 - \zeta_3^2 = 0 \} \subseteq \C^3,
\end{equation*}
where $\epsilon > 0$. In general, the input to the $\dbar$-equation will live on $X'$,
while the solutions are in general only defined on $X$.
For precise definitions of what we mean by $L^p$-forms and $C^0$ forms on $X'$ and $X$,
see Section \ref{sec:lp-forms}.

\begin{thm}\label{thm:main1}
    Let $\mathcal{K}$ be the integral operator from \cite{AS2} on $X'$,
    as here defined in \eqref{eq:AS1},
    and let $\frac{4}{3} < p\leq \infty$ and $q \in \{1,2\}$.
    Then:
    
    \medskip
    (i) $\mathcal{K}$ gives a bounded linear operator from $L^p_{0,q}(X')$ to $L^p_{0,q-1}(X)$.
    
    \medskip
    (ii) $\mathcal{K}$ gives a continuous linear operator from $L^\infty_{0,q}(X')$ to $C^0_{0,q-1}(X)$.
\end{thm}

In particular, one obtains the following result about the $\mathcal{A}$-sheaves from \cite{AS2}.

\begin{cor} \label{cor:asheaves}
    Let, as in \cite{AS2}, $\mathcal{A}^X_{q}$ be the sheaf of currents which can be written as finites sums of the 
    \begin{equation*}
        \xi_\nu\wedge (\mathcal{K}_\nu(\dots \xi_2 \wedge \mathcal{K}_2(\xi_1\wedge \mathcal{K}_1(\xi_1)))),
    \end{equation*}
    where each $\mathcal{K}_i$ is an integral operator as in Theorem \ref{thm:main1},
    and $\xi_i$ are smooth forms on $X'$. Then
    \begin{equation*}
        \mathcal{A}^X_{q} \subseteq C^0_{0,q}(X).
    \end{equation*}
\end{cor}

Although the Koppelman operator $\mathcal{K}$ maps $L^p_{0,q}(X')$ to $L^p_{0,q-1}(X)$ for $p > 4/3$,
this does not necessarily imply that the $\dbar$-equation is locally solvable in $L^p$ for $p > 4/3$,
since it is not necessarily the case that \eqref{eqkoppel} holds for $\varphi \in L^p$.
However, in order to describe when the Koppelman formula \eqref{eqkoppel} does indeed hold,
we first need to discuss various definitions of the $\dbar$-operator on $L^p$-forms
on singular varieties. If we let $\dbar_{sm}$ be the $\dbar$-operator on smooth $(0,q)$-forms
with support on $X^*=X\setminus\{0\}$ away from the singularity,
then this operator has various extensions as a closed operator in $L^p_{0,q}(X)$.

One extension of the $\dbar_{sm}$-operator is the maximal closed extension, i.e., the weak $\dbar$-operator $\dbar_w^{(p)}$
in the sense of currents, so if $g \in L^p_{0,q}(X)$, then $g \in \Dom \dbar_w^{(p)}$ if $\dbar g \in L^p_{0,q}(X)$
in the sense of distributions on $X$. When it is clear from the context, we will drop the superscript $(p)$ in $\dbar^{(p)}_w$,
and we will for example write $g \in \Dom \dbar_w \subset L^p_{0,q}$. For the $\dbar_w$-operator, we obtain the following
result about the Koppelman formula \eqref{eqkoppel}.

\begin{thm}\label{thm:main3}
Let $\mathcal{K}$ be the integral operator from Theorem \ref{thm:main1}.
Let $\varphi \in \Dom \dbar_w \subseteq L^p_{0,q}(X')$, where $2 \leq p \leq \infty$ and $q \in \{1,2\}$.

Then
\begin{eqnarray} \label{eq:dbarlp}
\varphi(z) &=& \dbar_w \mathcal{K}\varphi(z) + \mathcal{K}\big( \dbar_w \varphi\big)(z)
\end{eqnarray}
in the sense of distributions on $X$.
\end{thm}

Another extension of the $\dbar$-operator is the minimal closed extension, i.e., the
strong extension $\dbar_s^{(p)}$ of $\dbar_{sm}$, which is the graph
closure of $\dbar_{sm}$ in $L^p_{0,q}(X) \times L^p_{0,q+1}(X)$, so
$\varphi \in \Dom \dbar_s^{(p)} \subset L^p_{0,q}(X)$, if there exists a sequence of smooth 
forms $\{\varphi_j\}_j \subset L^p_{0,q}(X)$ with support away from the singularity, i.e.,
$$\supp \varphi_j \cap \{0\} = \emptyset,$$
such that
\begin{eqnarray}\label{eq:dbars1}
\varphi_j \rightarrow \varphi \ \ \ &\mbox{ in }& \ \ L^p_{0,q}(X),\\
\dbar \varphi_j \rightarrow \dbar \varphi \ \ \ &\mbox{ in }& \ \ L^p_{0,q+1}(X)\label{eq:dbars2}
\end{eqnarray}
as $j\rightarrow \infty$.

On smooth varieties, these extensions coincide by Friedrichs' extension lemma, see for example
\cite{LiMi}*{Theorem~V.2.6}. From our results below, it will follow that in $L^2$ on the
$A_1$-singularity, the $\dbar_w$ and $\dbar_s$ operators do indeed coincide. 
In $L^p$ for more general $p$, it is not clear to us whether the $\dbar_w$ and $\dbar_s$ operators
still coincide on the $A_1$-singularity. On other varieties, one can however write explicitly
examples of functions which are in $\Dom \dbar_w$, but not in $\Dom \dbar_s$, even in $L^2$.

\begin{ex}
Let $Z$ be the cusp
\begin{equation*}
Z = \{ (z,w) \in B_1(0) \mid z^3-w^2 = 0 \} \subseteq \C^2.
\end{equation*}
Then, using the normalization $\pi : t \mapsto (t^2,t^3)$
of $Z$, one can verify that the function $\varphi = z/w$ is in $L^2(Z)$, and $\varphi$ is $\dbar$-closed, so
$\varphi \in \Dom \dbar_w \subseteq L^2(Z)$. By \cite{RuSerre}*{Theorem~1.2}, the kernel of the $\dbar_s$-operator on
$\Dom \dbar_s \subseteq L^2(Z)$ is exactly $\widehat{\Ok}(Z)$, the ring of weakly holomorphic functions on $Z$.
Thus, if $\varphi \in \Dom \dbar_s$, we would thus get that $\varphi \in \widehat{\Ok}(Z)$ since $\dbar \varphi = 0$.
However, since $\pi^* \varphi = 1/t$, one gets that $\varphi$ is not locally bounded near $0$, so it is not
weakly holomorphic, and thus, $\varphi \notin \Dom \dbar_s \subseteq L^2(Z)$, but $\varphi \in \Dom \dbar_w \subseteq L^2(Z)$.
\end{ex}

For the strong $\dbar$-operator, we obtain the following.

\begin{thm}\label{thm:main4}
Let $\mathcal{K}$ be the integral operator from Theorem \ref{thm:main1}
and let $\varphi\in \Dom \dbar_w \subseteq L^2_{0,q}(X')$, $1\leq q \leq 2$. Then
\begin{eqnarray*}
\mathcal{K} \varphi &\in& \Dom\dbar_s \subset L^2_{0,q-1}(X).
\end{eqnarray*}
\end{thm}

Since $\mathcal{K}$ maps $\Dom \dbar_w \to \Dom \dbar_s$, and $\dbar$ maps
$\Dom \dbar_w \to \Dom \dbar_w$ and $\Dom \dbar_s \to \Dom \dbar_s$, we get as a
corollary of Theorem~\ref{thm:main3} and Theorem~\ref{thm:main4} the following.

\begin{cor}\label{cor:main4}
In $L^2$ on the $A_1$-singularity, the $\dbar_s$ and $\dbar_w$ operators coincide.
\end{cor}

\medskip

The setting in \cite{AS2} is rather different compared to this article, since here, we are mainly concerned
with forms on $X$ with coefficients in $L^p$, while in \cite{AS2}, the type of forms
considered, denoted $\mathcal{W}_q(X)$, are generically smooth, and with in a certain sense ``holomorphic singularities''
(like for example the principal value current $1/f$ of a holomorphic function $f$),
but there is no direct growth condition on the singularities.
For the precise definition of the class $\mathcal{W}_q(X)$, we refer to \cite{AS2}. In the setting of \cite{AS2},
the $\dbar$-operator $\dbar_X$ considered there is different from the ones considered
here, $\dbar_s$ and $\dbar_w$. For currents in $\mathcal{W}_q(X)$, one can define the product with certain ``structure forms''
$\omega_X$ associated to the variety. A current $\mu \in \mathcal{W}_q(X)$ lies in $\Dom \dbar_X$ if there exists a current
$\tau \in \mathcal{W}_{q+1}(X)$ such that $\dbar (\mu \wedge \omega) = \tau \wedge \omega$ for all structure forms $\omega$.
(To be precise, this formulation works when $X$ is Cohen-Macaulay, as is the case for example here, when $X$ is a hypersurface).

Combining our results about $\mathcal{K}$ and the $\dbar_w$- and $\dbar_s$-operator
with some properties about the $\mathcal{W}_X$-sheaves,
we obtain results similar to Theorem~\ref{thm:main4} for the $\dbar_X$-operator,
answering in part a question in \cite{AS2} (see the paragraph at the end of page 288 in \cite{AS2}).

\begin{thm}\label{thm:main5}
Let $\mathcal{K}$ be the integral operator from Theorem \ref{thm:main1}
and let $\varphi\in \Dom \dbar^{(2)}_w \cap \mathcal{W}_q(X')$, $1\leq q \leq 2$. Then
\begin{eqnarray*}
\mathcal{K} \varphi &\in& \Dom\dbar_X.
\end{eqnarray*}
\end{thm}

For a hypersurface $X$, any structure form is an invertible holomorphic function times
the Poincar\'e-residue of $d\zeta_1 \wedge d\zeta_2 \wedge d\zeta_3/h$, where $h$
is the defining function of $X$. In our case, $h(\zeta) = \zeta_1 \zeta_2 - \zeta_3^2$, and 
the Poincar\'e residue $\omega_X$ can be defined for example as
\begin{equation*}
    \omega_X = \left.\frac{dz_1 \wedge dz_2}{-2\zeta_3} \right|_{X},
\end{equation*}
which one can verify lies in $L^2_{2,0}(X)$.
The conclusions of Theorem~\ref{thm:main5} means that
\begin{equation} \label{eq:thm5concl}
    \dbar (\mathcal{K} \varphi \wedge \omega_X) = (\dbar \mathcal{K} \varphi) \wedge \omega_X.
\end{equation}
Since $\varphi \in \Dom \dbar_w \subseteq L^2(X')$, by the Koppelman formula for $\dbar_w$ on $L^2$,
we get that $\dbar K \varphi \in L^2(X)$. Thus, since $\omega_X \in L^2_{loc}(X)$, the products $\mathcal{K}\varphi \wedge \omega_X$
and $(\dbar \mathcal{K} \varphi) \wedge \omega_X$ exist (almost-everywhere) pointwise and lie in $L^1_{loc}(X)$.

\medskip

The results of the present paper have to a large extent been generalized in \cite{LR2} to so-called affine cones over smooth projective
complete intersections, which in particular include the $A_1$-singularity. The methods used in \cite{LR2},
which rely on estimates directly on the variety, are rather different to the methods here, which rely
on estimates on a finite branched covering. In addition to the fact that we obtain here stronger results in 
Theorem~\ref{thm:main4} and as a consequence also stronger results in Corollary~\ref{cor:main4} and
Theorem~\ref{thm:main5}, compared to the results in \cite{LR2} on the $A_1$-singularity, we also believe that
the techniques used in this article might still be of interest when trying to extend our results to more
general varieties. In particular, in preliminary work about Koppelman formulas on surfaces with canonical
singularities, which include the $A_1$-singularity, it appears that a combination of these two techniques
is useful.

\medskip

The $A_1$-singularity has in many ways very mild singularities, and one way which this manifests
itself is that it satisfies the conditions for being treated in almost all articles about the
solvability of the $\dbar$-equation on singular varieties in recent years. 

The following results about that the $\dbar$-equation $\dbar f = g$ is solvable on the $A_1$-singularity can be found in earlier works.

\begin{itemize}
    \item $f \in C_{0,q}^\infty(X^*)$ if $g \in C_{0,q-1}^{\infty}(X')$ is treated in \cite{HePo}.
    \item $f \in C_{0,1}^\alpha(X)$ for $\alpha < 1/2$ if $g \in L^\infty_{0,1}(X')\cap C^0(X')$ is treated in \cite{FoGa}.
    \item $f \in C_{0,q}^{1/2}(X)$ if $g \in L^\infty_{0,q}(X)$ is treated in \cite{RuppThesis}.
    \item $f \in C_{0,1}^{\alpha}(X)$ for $\alpha < 1$ if $g \in L^\infty_{0,1}(X')$ and $g$ has compact support is treated in \cite{RuZeI}.
    \item $f \in L^p_{0,1}(X)$ for $p > 4/3$ if $g \in L^p_{0,1}(X)$ is treated in \cite{RuMatZ2}, where the $\dbar$-operator considered is the
        $\dbar_w$-operator. In addition, it is shown that for $1 \leq p < 4/3$, the $\dbar_w$-cohomology in $L^p$ is non-zero.
    \item $f \in L^2_{0,q}(X)$ if $g \in L^2_{0,q-1}(X)$ is treated in \cite{RuSerre}, where the $\dbar$-operator considered is the $\dbar_s$-operator.
\end{itemize}

Note that here we just refer to the results concerning the $A_1$-singularity in those articles, while all
the articles treat results about the $\dbar$-equation on other varieties as well.

\medskip

This paper is organized as follows. In Section~\ref{sec:covering}, we describe a $2$-sheeted covering
of the $A_1$-singularity, relations between $L^p$-forms on $X$ and on the covering, and describe various integral
estimates on this covering. In Section~\ref{sec:main1}, we recall how the Koppelman operators from \cite{AS2}
are constructed, and prove the first main result, Theorem~\ref{thm:main1}. In Section~\ref{sec:main2}, we prove
an estimate for a cut-off procedure, Theorem~\ref{thm:main2}, which is then used in the proof of Theorem~\ref{thm:main3},
about the $\dbar_w$-operator. In Section~\ref{sec:main3}, we then prove Theorem~\ref{thm:main4}, about the $\dbar_s$-operator,
and Theorem~\ref{thm:main5}, about the $\dbar_X$-operator. Finally, in Appendix~\ref{sec:appendix}, we collect
various integral kernel estimates on $\C^n$, which we have made use of in Section~\ref{sec:covering} for obtaining integral estimates
on the $2$-sheeted covering.

\bigskip
{\bf Acknowledgments.}
This research was supported by the Deutsche Forschungsgemeinschaft (DFG, German Research Foundation), 
grant RU 1474/2 within DFG's Emmy Noether Programme.
The first author was partially supported by the Swedish Research Council.

\bigskip
\section{The $2$-sheeted covering of the $A_1$-singularity} \label{sec:covering}

\smallskip
\subsection{Some notation}

Let us recall shortly that
we will consider the variety defined by $\{ g(\zeta) = 0 \}$, where $g(\zeta) = \zeta_1 \zeta_2 - \zeta_3^2$,
on two different balls in $\C^3$.
We let $D = B_1(0) \subseteq \C^3$ and $D' = B_{1+\epsilon}(0) \subseteq \C^3$ for some $\epsilon > 0$, and we define:
\begin{equation*}
    X = \{ \zeta \in D \mid g(\zeta) = 0 \} \text{ and } X' = \{ \zeta \in D' \mid g(\zeta) = 0 \}.
\end{equation*}
Note that $X$ and $X'$ can be covered by the $2$-sheeted covering map 
$$\pi: (w_1,w_2) \mapsto (w_1^2,w_2^2,w_1w_2),$$
which is branched just in the origin.
Let
\begin{equation*}
\tilde{D} := \pi^{-1}(D)\ \mbox{ and }\ \tilde{D}' = \pi^{-1}(D').
\end{equation*}
In this section, we consider the $2$-sheeted covering maps
$\pi : \tilde{D} \to X$ and $\pi : \tilde{D}' \to X'$, respectively.
We will use this covering to estimate the integral operators of Andersson--Samuelsson
by use of certain integral estimates in $\C^2$ which are adopted to our particular situation.
Basic estimates in $\C^n$ which are needed are postponed to Appendix~\ref{sec:appendix},
for convenience of the reader.

\smallskip
\subsection{Pullback of $\|\eta\|^2=\|\zeta-z\|^2$}

Here, we prove an estimate of how the pullback of $\|\eta\|^2$ to the covering
behaves, where $\eta = \zeta-z$, which will be fundamental in obtaining our estimates for the pullback
of the integral kernels.

We will as above let $w=(w_1,w_2)$ in the covering correspond to the $\zeta$-variables on $\C^3$
by 
$$\pi(w_1,w_2) = (w_1^2,w_2^2,w_1w_2) = \zeta,$$
and we will let $x=(x_1,x_2)$ correspond to the $z$-variables on $\C^3$, i.e., 
$$\pi(x_1,x_2) = (x_1^2,x_2^2,x_1 x_2) = z.$$
We let
\begin{equation*}
    \alpha^2 = \pi^* \|\eta\|^2 = |w_1^2-x_1^2|^2 + |w_2^2-x_2^2|^2 + |w_1w_2-x_1x_2|^2,
\end{equation*}
and
\begin{equation*}
    \beta_-^2 = |w_1-x_1|^2 + |w_2-x_2|^2=\|w-x\|^2,
\end{equation*}
and
\begin{equation*}
\beta_+^2 = |w_1+x_1|^2 + |w_2+x_2|^2=\|w+x\|^2.
\end{equation*}

\begin{lma} \label{lmanorm}
    \begin{equation*}
        \alpha^2 \leq \beta_+^2\beta_-^2 \leq 4\alpha^2.
    \end{equation*}
\end{lma}

\begin{proof}
    Using the parallelogram identity 
    $$|a-b|^2 + |a+b|^2 = 2(|a|^2 + |b|^2)$$
    we get
    \begin{eqnarray*}
        \beta_+^2\beta_-^2 &=& |w_1^2-x_1^2|^2 + |w_2^2-x_2^2|^2 + |(x_1-w_1)(x_2+w_2)|^2\\
        && + |(x_1+w_1)(x_2-w_2)|^2 \\
        &=& |w_1^2-x_1^2|^2 + |w_2^2-x_2^2|^2 + |(x_1x_2-w_1w_2) + (x_1w_2-w_1x_2)|^2\\
        && + |(x_1x_2-w_1w_2)-(x_1w_2-w_1x_2)|^2 \\
        &=&  |w_1^2-x_1^2|^2 + |w_2^2-x_2^2|^2 + 2|x_1x_2-w_1w_2|^2 + 2|x_1w_2-w_1x_2|^2\\
        &\geq& |w_1^2-x_1^2|^2 + |w_2^2-x_2^2|^2 + |w_1w_2-x_1x_2|^2 = \alpha^2,
    \end{eqnarray*}
    so the first inequality is proved.
    To prove the second inequality, we note that by the equality
\begin{align*}
    \beta_+^2\beta_-^2 = |w_1^2-x_1^2|^2 + |w_2^2-x_2^2|^2 + 2|x_1x_2-w_1w_2|^2 + 2|x_1w_2-w_1x_2|^2
\end{align*}
    from the equation above, it is enough to prove that
    \begin{equation*}
        |x_1w_2-w_1x_2|^2 \leq \alpha^2.
    \end{equation*}
    To prove this, we use the triangle inequality and the inequality $|ab| \leq (1/2)(|a|^2 + |b|^2)$:
    \begin{eqnarray*}
        |x_1w_2-w_1x_2|^2 &=& |(x_1w_2-w_1x_2)^2| = |x_1^2w_2^2 + w_1^2x_2^2 - 2x_1w_2w_1x_2| \\
        &=& |x_1^2w_2^2 + w_1^2x_2^2 - w_1^2w_2^2 - x_1^2x_2^2 + w_1^2w_2^2 + x_1^2x_2^2 - 2x_1w_2w_1x_2| \\
        &=& |(w_1^2-x_1^2)(x_2^2-w_2^2) + (w_1w_2-x_1x_2)^2|  \\
        &\leq& (1/2)|w_1^2-x_1^2|^2 + (1/2)|w_2^2-x_2^2|^2 + |(w_1w_2-x_1x_2)^2|  \\
        &\leq& |w_1^2-x_1^2|^2 + |w_2^2-x_2^2|^2 + |w_1w_2-x_1x_2|^2.
    \end{eqnarray*}
\end{proof}

\smallskip
\subsection{Integral kernel estimates for the covering}

We will now provide fundamental integral estimates for the pull-back under $\pi$
of the principal parts of the integral formulas of Andersson--Samuelsson.
Let $dV(w)$ and $dV(x)$ denote the standard Euclidean volume forms on $\C^2_{w}$ and $\C^2_{x}$.
We denote the different coordinates of $\C^2$ by the variables $w=(w_1,w_2)$ and $x=(x_1,x_2)$.

\begin{lma} \label{lmakerninteg}
    Let $K$ be an integral kernel on ${\tilde{D}'}\times {\tilde{D}} \subset \subset \C_w^2 \times \C_x^2$ of the form
   \begin{equation*}
       K(w,x) = \frac{|f|}{\alpha^3},
   \end{equation*}
   where $f$ is one of the functions $w_1^2,w_2^2,w_1w_2,x_1^2, x_2^2, x_1x_2$. 
   Let $\gamma > -6$ if $f\in \{w_1^2,w_1w_2,w_2^2\}$ and $\gamma>-4$ if $f\in \{x_1^2,x_1x_2,x_2^2\}$.
   Then there exists a constant $C_\gamma>0$ such that
   \begin{align}
       \label{eqkerninteg1} I_1(x) &:= \int_{\tilde{D}'} \|w\|^\gamma K(w,x) dV(w) \leq C_\gamma
       \left\{\begin{array}{ll}
1 & \ ,\ \gamma>0,\\
1+ \big|\log \|x\|\big| &\ ,\ \gamma=0,\\
\|x\|^\gamma &\ ,\ \gamma<0,
\end{array}\right.
   \end{align}
   for all $x\in\ \tilde{D}$ with $x\neq 0$.
\end{lma}

\begin{proof}
We know by Lemma \ref{lmanorm} that $\alpha \gtrsim \|w-x\| \|w+x\|$,
and so
$$|I_1(x)| \lesssim \|x\|^{2-\delta} \int_{\tilde{D}'} \frac{\|w\|^{\delta+\gamma} dV(w)}{\|w-x\|^3 \|w+x\|^3},$$
where $\delta\in\{0, 2\}$. So, the assertion follows from the basic estimate, Lemma \ref{lem:estimateCn3},
by considering the different cases separately.
\end{proof}

\bigskip
By an elaboration of the argument of the generalization of Young's inequality for convolution
integrals in \cite{Range}*{Appendix B}, we then get the following lemma.

\begin{lma} \label{lmacoveringintegr}
    Let $\mathcal{K}$ be an integral operator defined by
    \begin{equation*}
        \mathcal{K} \varphi(x) = \int K(w,x) \varphi(w) dV(w),
    \end{equation*}
    acting on forms on $\tilde{D}'$ and returning forms on $\tilde{D}$, where $K$ is of the form
    \begin{equation*}
        K = \frac{gf}{\alpha^3},
    \end{equation*}
    where $g \in L^\infty({\tilde{D}'}\times {\tilde{D}})$ and $f$ is one of $w_1^2, w_1w_2, w_2^2, x_1^2, x_1x_2, x_2^2$.
    
   \medskip
    (i) Let $\frac{4}{3} < p \leq \infty$.
    Then $\mathcal{K}$ maps $\|w\|^{2-4/p} L^p({\tilde{D'}})$ continuously to $\|x\|^{-4/p} L^p({\tilde{D}})$, i.e., if $\|w\|^{4/p-2} \varphi \in L^p({\tilde{D'}})$, then
    $\|x\|^{4/p} \mathcal{K} \varphi \in L^p({\tilde{D}})$.
    
    \medskip
    (ii) Assume that $\|w\|^{-2} \varphi \in L^\infty({\tilde{D}'})$ and that    
    $\lim_{x\rightarrow 0} g(\cdot,x) =  g(\cdot,0)$ in $L^r(\tilde{D}')$ for some $r>2$.
    Then $\mathcal{K} \varphi$ is continuous at the origin.
\end{lma}

\begin{proof}
    (i) Let us first consider the case $p<\infty$. Choose 
    $$q:=p/(p-1)\ \ \ \mbox{ and } \ \ \ \eta:=2-4/p.$$
    So, $1/p+1/q=1$ and
    $$\gamma:=\eta q = (2p-4)/(p-1)= 2 + \frac{2}{1-p}> -4$$ 
    (because of $p>4/3$). 

    We want to show that the $L^p$-norm of $\|x\|^{4/p} \mathcal{K} \varphi$ is finite, and
    we begin by estimating and decomposing, and using the H\"older inequality (with $1/p+1/q=1$) in the following way:
   
    \begin{align*}
        I := &\int_{\tilde{D}} \|x\|^{4} \left| \int_{\tilde{D}'} \frac{gf \varphi}{\alpha^3} dV(w) \right|^p dV(x) \\ \leq
        &\int_{\tilde{D}} \|x\|^4 \left( \int_{\tilde{D}'} \left(\frac{|gf|\big|\|w\|^{4/p-2}\varphi\big|^p}{\alpha^{3}}\right)^{1/p}
        \left(\frac{|gf| (\|w\|^{2-4/p})^q}{\alpha^{3}}\right)^{1/q} dV(w) \right)^{p} dV(x) \\ \leq
        &\int_{\tilde{D}} \|x\|^4 \int_{\tilde{D}'} \frac{|gf|\big|\|w\|^{4/p-2}\varphi\big|^p}{\alpha^{3}} dV(w) 
        \left(\int_{\tilde{D}'} \frac{|gf|\|w\|^{\eta q}}{\alpha^{3}}\right)^{p/q} dV(w) dV(x).
    \end{align*}
    
    From now on, let us just consider the situation that $\gamma=\eta q <0$.
     The other cases, $\gamma=0$ and $\gamma>0$, respectively,
    are even simpler: just replace $\|x\|^\gamma$ in the following by $1+\big|\log\|x\|\big|$ or $1$, respectively.
    Using \eqref{eqkerninteg1} on the second inner integral, $\gamma p/q = \eta p = 2p -4$ and Fubini's Theorem
    one obtains
    \begin{align*}
        I \lesssim & \int_{\tilde{D}} \int_{\tilde{D}'} \|x\|^4 \frac{|gf|\big| \|w\|^{4/p-2}\varphi\big|^p}{\alpha^{3}} dV(w) \|x\|^{\gamma p/q}dV(x)\\
        = & \int_{\tilde{D}'} \big| \|w\|^{4/p-2}\varphi\big|^p
        \int_{\tilde{D}} \|x\|^{2p} \frac{|gf|}{\alpha^{3}} dV(x) dV(w)
    \end{align*}
    By use of \eqref{eqkerninteg1}, we then get that
    \begin{equation*}
        I \lesssim \int_{\tilde{D}'} \big|\|w\|^{4/p-2}\varphi\big|^p dV(w) = \big\| \|w\|^{4/p-2}\varphi\big\|^p_{L^p({\tilde{D}'})} < \infty.
    \end{equation*}
    
    It remains to consider the case $p=\infty$ which is even simpler:
    \begin{eqnarray*}
    \left| \int_{\tilde{D}'} \frac{gf \varphi}{\alpha^3} dV(w) \right| 
    \leq \big\| \|w\|^{-2} \varphi \big\|_\infty \int_{\tilde{D}'} \|w\|^2 \frac{|gf|}{\alpha^3} dV(w) \lesssim \big\| \|w\|^{-2} \varphi \big\|_\infty
     \end{eqnarray*}
    by use of Lemma \ref{lmakerninteg}.

    \bigskip
    (ii) If $f\in\{x_1^2,x_1x_2,x_2^2\}$, then 
    $$\mathcal{K}\varphi(x) = f(x) \int \frac{g \varphi dV(w)}{\alpha^3}.$$
    But
    $$\left|\int \frac{g \varphi dV(w)}{\alpha^3}\right| \lesssim \int \frac{\|w\|^2}{\alpha^3} \lesssim \log\|x\|$$
    by Lemma \ref{lmakerninteg}. So, it follows that $\mathcal{K}\varphi$ is continuous at $0\in\C^2$
    with $\mathcal{K}\varphi(0,0)=0$.
    
    \medskip
    It remains to treat the case $f\in\{w_1^2,w_1w_2,w_2^2\}$.
    We know from part (i) that $\mathcal{K}\varphi$ is a bounded function (the integral exists for all $x=(x_1,x_2)$).
    Let $c_\varphi := \big| \|w\|^{-2} \varphi\|_\infty$. Using this and $|f| \leq \|w\|^2$, we get
    
    \begin{eqnarray*}
    \Delta(x) := \left| \mathcal{K}\varphi(x) - \mathcal{K}\varphi(0)\right| \leq
    c_\varphi \int \|w\|^4 \left| \frac{g(w,x)}{\alpha^3(w,x)} - \frac{g(w,0)}{\alpha^3(w,0)} \right| dV(w).
    \end{eqnarray*}
    
    Using
    $$\alpha^3 \sim \delta(w,x) := \|w-x\|^3 \|w+x\|^3$$
    from Lemma \ref{lmanorm}, we have
     \begin{eqnarray*}
    \Delta(x)  &\lesssim&
    \int \|w\|^4 \left| \frac{g(w,x)}{\|w-x\|^3\|w+x\|^3} - \frac{g(w,0)}{\|w\|^6} \right| dV(w)\\
    &=& \int \|w\|^4 \left| \frac{\|w\|^6 g(w,x) -\delta(w,x) g(w,0)}{\delta(w,x) \|w\|^6} \right| dV(w)\\
    &\leq&  \int \|w\|^4 \left| \frac{\|w\|^6 g(w,x) -\delta(w,x) g(w,x)}{\delta(w,x) \|w\|^6} \right| dV(w)\\
    && + \int \|w\|^4 \left| \frac{\delta(w,x) g(w,x) -\delta(w,x) g(w,0)}{\delta(w,x) \|w\|^6} \right| dV(w).
        \end{eqnarray*}
        
    By use of the Taylor expansion, we have
    $$\left| \delta(w,x) - \|w\|^6 \right| = \left| \delta(w,x) - \delta(w,0) \right| \lesssim \sum_{k=1}^6 \|x\|^k \|w\|^{6-k}.$$
    This gives
    \begin{eqnarray*}
    \Delta_1(x) &:= & \int \|w\|^4 \left| \frac{\|w\|^6 g(w,x) -\delta(w,x) g(w,x)}{\delta(w,x) \|w\|^6} \right| dV(w)\\
    &\lesssim& \|x\| \sum_{k=1}^6 \|g\|_\infty \int \frac{\|x\|^{k-1} \|w\|^{4-k}}{\delta(w,x)} dV(w)
    \lesssim \|x\| \|g\|_\infty,
    \end{eqnarray*}
    where we have used Lemma \ref{lmakerninteg} for the last step.
    
    On the other hand,
    \begin{eqnarray*}
    \Delta_2(x) &:=&  \int \|w\|^4 \left| \frac{\delta(w,x) g(w,x) -\delta(w,x) g(w,0)}{\delta(w,x) \|w\|^6} \right| dV(w)\\
    &=& \int \left| \frac{g(w,x) - g(w,0)}{\|w\|^2} \right| dV(w).
    \end{eqnarray*}
    Let $s=\frac{r}{r-1}$. Then $s<2$ and the H\"older inequality gives:
    \begin{eqnarray*}
    \Delta_2(x) &\leq& \left( \int \frac{dV(w)}{\|w\|^{2s}}\right)^{1/s} \|g(\cdot,x) - g(\cdot,0)\|_{L^r(\tilde{D}')}\\
    &\lesssim& \|g(\cdot,x) - g(\cdot,0)\|_{L^r(\tilde{D}')} \rightarrow 0
    \end{eqnarray*}
    as $x\rightarrow 0$ by assumption.
    
    Summing up, we have $\Delta(x)=\Delta_1(x)+\Delta_2(x) \rightarrow 0$ as $x\rightarrow 0$.
\end{proof}

Let us remark that the estimates in the proof of Lemma \ref{lmacoveringintegr} are pretty rough.
We could do much better, but the Lemma -- as it stands -- is sufficient for our purpose and better estimates
would complicate the presentation considerably. 
In the special case $p=2$, we will need some better estimates which we give in the next section.

\smallskip
\subsection{Estimates for cut-off procedures}

In the proof of the homotopy formula for the $\dbar_w$ and the $\dbar_s$-operator,
we will use certain cut-off procedures. For these we require some better estimates
which will be given in this section.
For $k\in\Z$, $k\geq 1$, let
\begin{equation} \label{eq:Dk}
    \tilde{D}_k := \{ x\in \tilde{D}': e^{-e^{k+1}/2} < \|x\| < \sqrt{2}e^{-e^k/2} \}.
\end{equation}
A simple calculation shows that since $k\geq 1$,
\begin{equation} \label{eq:Dkprime}
\tilde{D}_k \subset \tilde{D}_k' := \{ x\in \tilde{D}': e^{-e^{k+1}/2} < \|x\| < e^{-e^{k-1}/2} \},
\end{equation}
so in the following proofs of the following lemmas, we can consider integration over $\tilde{D}_k'$ instead.

\begin{lma} \label{lem:cutoff1}
    Let $\mathcal{K}$ be an integral operator defined by
    \begin{equation*}
        \mathcal{K} \varphi(x) = \int_{\tilde{D}'} |K(w,x)| \varphi(w) dV(w),
    \end{equation*}
    where $K$ is of the form
    \begin{equation*}
        K = \frac{gf}{\alpha^3},
    \end{equation*}
    where $g \in L^\infty({\tilde{D}'}\times {\tilde{D}})$ and $f\in\{w_1^2,w_1w_2,w_2^2, x_1^2,x_1x_2,x_2^2\}$.
    
    Then there exists a constant $C>0$ such that
    \begin{eqnarray*}
    \int_{\tilde{D}_k} \frac{|\mathcal{K}\varphi (x)|^2}{\log^2 \|x\|} dV(x) &<& C \|\varphi\|^2_{L^2(\tilde{D}')}
    \end{eqnarray*}
    for all $\varphi\in L^2(\tilde{D}')$ and all $k\geq 1$.
 \end{lma}

\begin{proof}
We proceed similarly as in the proof of Lemma \ref{lmacoveringintegr},
but need to estimate the integrals more carefully:
    \begin{align*}
        I_k := &\int_{\tilde{D}_k} \frac{1}{\log^2\|x\|} \left| \int_{\tilde{D}'} \frac{|gf| \varphi}{\alpha^3} dV(w) \right|^2 dV(x) \\
        \leq &\int_{\tilde{D}_k} \frac{1}{\log^2\|x\|} \left( \int_{\tilde{D}'} \left(\frac{|gf| |\varphi|^2}{\alpha^{3}}\right)^{1/2}
        \left(\frac{|gf|}{\alpha^{3}}\right)^{1/2} dV(w) \right)^{2} dV(x) \\ \leq
        &\int_{\tilde{D}_k} \frac{1}{\log^2\|x\|} \left( \int_{\tilde{D}'} \frac{|gf| |\varphi|^2}{\alpha^{3}} dV(w) \right)
        \left(\int_{\tilde{D}'} \frac{|gf|}{\alpha^{3}} dV(w)\right) dV(x).
    \end{align*}
    Applying Lemma \ref{lmakerninteg} to the second inner integral gives
    \begin{align*}
    I_k \lesssim & \int_{\tilde{D}_k} \frac{1}{\big|\log\|x\|\big|} \int_{\tilde{D}'} \frac{|gf| |\varphi|^2}{\alpha^{3}} dV(w) dV(x).
    \end{align*}
    
Using $\alpha \gtrsim \|w-x\| \|w+x\|$ (see Lemma \ref{lmanorm}), Fubini's Theorem and the fact that 
$$\int_{\tilde{D}_k} \frac{|fg|  dV (x)}
{\|w-x\|^3 \|w+x\|^3 \big|\log\|x\|\big|} \lesssim 1$$
for all $w\in\C^2$ by Lemma \ref{lem:estimateCn4} (with $\gamma\in\{4,6\}$) together with \eqref{eq:Dkprime},
we finally obtain
$$I_k \lesssim \int_{\tilde{D}'} |\varphi|^2 dV(w) = \|\varphi\|^2_{L^2(\tilde{D}')}.$$
\end{proof}

Another cut-off estimate that we will need is:

\begin{lma} \label{lem:cutoff2}
    For $k\in \Z$, $k\geq 1$,
    let $\mathcal{K}^k$ be integral operators defined by
    \begin{equation*}
        \mathcal{K}^k \varphi(x) = \int_{\tilde{D}_k} |K(w,x)| \frac{\varphi(w)}{\|w\|^2 \big|\log\|w\|\big|} dV(w),
    \end{equation*}
    where $K$ is of the form
    \begin{equation*}
        K = \frac{gf}{\alpha^3},
    \end{equation*}
    where $g \in L^\infty({\tilde{D}'}\times {\tilde{D}})$ and $f\in\{w_1^2,w_1w_2,w_2^2,x_1^2,x_1x_2,x_2^2\}$.
    
    Let $\varphi\in L^2(\tilde{D}')$. Then
    \begin{eqnarray*}
    \int_{\tilde{D}} \|x\|^4 |\mathcal{K}^k\varphi (x)|^2 dV(x) \longrightarrow 0
    \end{eqnarray*}
    for $k\rightarrow \infty$.
 \end{lma}

\begin{proof}
We proceed similarly as in the proof of Lemma \ref{lem:cutoff1}:
    \begin{align*}
        I_k := &\int_{\tilde{D}} \|x\|^4 \left| \int_{\tilde{D}_k} \frac{|gf| \varphi}{\alpha^3 \|w\|^2 \big|\log\|w\|\big|} dV(w) \right|^2 dV(x) \\ \leq
        &\int_{\tilde{D}} \|x\|^4 \left( \int_{\tilde{D}_k} \left(\frac{|gf| |\varphi|^2}{\alpha^{3}\big|\log\|w\|\big|}\right)^{1/2}
        \left(\frac{|gf|}{\alpha^{3}\|w\|^4\big|\log\|w\|\big|}\right)^{1/2} dV(w) \right)^{2} dV(x) \\ \leq
        &\int_{\tilde{D}}  \left(\int_{\tilde{D}_k} \frac{|gf| |\varphi|^2}{\alpha^{3} \big|\log \|w\|\big|} dV(w) \right)
        \|x\|^4 \left(\int_{\tilde{D}_k} \frac{|gf|}{\alpha^{3}\|w\|^4|\log\|w\|\big|} dV(w)\right) dV(x).
    \end{align*}
    Using $\alpha \gtrsim \|w-x\| \|w+x\|$ (see Lemma \ref{lmanorm}), Fubini's Theorem and the fact that 
$$\|x\|^4 \int_{\tilde{D}_k} \frac{|fg|  dV (w)}
{\|w-x\|^3 \|w+x\|^3 \|w\|^4 \big|\log\|w\|\big|} \lesssim 1$$
for all $x\in\C^2$ by Lemma \ref{lem:estimateCn4} (with $\gamma\in\{0,2\}$) and \eqref{eq:Dkprime}, we obtain
$$I_k \lesssim \int_{\tilde{D}_k} \frac{|\varphi|^2}{\big|\log\|w\|\big|} 
\left( \int_{\tilde{D}} \frac{|fg|}{\alpha^3} dV(x) \right) dV(w).$$
But now we can apply Lemma \ref{lmakerninteg} to the inner integral to conclude finally:
$$I_k \lesssim \int_{\tilde{D}_k} \frac{|\varphi|^2}{\big|\log\|w\|\big|} \big( 1 + \big|\log\|w\|\big|\big) dV(w) 
\lesssim \|\varphi\|^2_{L^2(\tilde{D}_k)} \rightarrow 0$$
for $k\rightarrow \infty$ as the domain of integration vanishes.
\end{proof}

\smallskip
\subsection{$L^p$-norms on the variety and the covering}\label{sec:lp-forms}

Let $1\leq p\leq \infty$.

When we consider a $L^p$-differential form as input into an integral operator,
it will be convenient to represent it in a certain ``minimal'' manner.
If $\varphi$ is a $(0,q)$-form on $X$ (or $X'$, respectively), then by \cite{RuppThesis}*{Lemma~2.2.1}, we can
write $\varphi$ uniquely in the form
\begin{equation}\label{eq:minrep}
    \varphi = \sum_{|I| = q} \varphi_I d\bar{z}_I,
\end{equation}
where $|\varphi|^2 (p)= \sqrt{2}^{q} \sum |\varphi_I|^2(p)$ in each regular point $p\in \Reg X$.
The constants here stem from the fact that $|d\overline{z_j}|=\sqrt{2}$ in $\C^n$.
In particular, we then get that $\varphi \in L^p_{0,q}(X)$ if and only if $\varphi_I \in L^p(X)$ for all $I$.
Note that the singular set of $X$ is negligible as it is a zero set.

\medskip
We say that $\varphi$ is continuous at a point $p\in X$ if there is a representation \eqref{eq:minrep}
such that all the coefficients $\varphi_I$ are continuous at the point $p$. This does not need to be the minimal representation.
Let $C^0_{0,q}(X)$ be the space of continuous $(0,q)$-forms on $X$.
$C^0_{0,q}(X)$ is a Fr\'echet space with the metric induced by the semi-norms $\|\cdot\|_{L^\infty, K_j}$,
where $K_1 \subset K_2 \subset K_3 \subset ...$ is a compact exhaustion of $X$.

We also note that continuous forms on $X$ have a continuous extension to a neighborhood of $X$ in $\C^3$
by the Tietze extension theorem.

\medskip
We let $dV_X$ be the induced volume form $i^*\omega^2/2$ on $X$, where $i : X \to \C^3$ is the inclusion and
$\omega$ is the standard Kähler form on $\C^3$. Then
\begin{equation} \label{eq:volform}
    \pi^* dV_X = 2(\|w_1\|^4 + \|w_2\|^4 + 4\|w_1w_2\|^2) dV(w).
\end{equation}
If we let
\begin{equation*}
    \xi^2 := \|w_1\|^2 + \|w_2\|^2 \text{ and } \tilde{\varphi}_I := \pi^* \varphi_I,
\end{equation*}
then since $2(\|w_1\|^4 + \|w_2\|^4 + 4\|w_1 w_2\|^2) \sim \xi^4$,  we get that
$\varphi \in L^p_{0,q}(X)$ if and only if $\xi^{4/p} \tilde{\varphi}_I \in L^p(\tilde{D})$,
where $\varphi$ is given in the minimal representation \eqref{eq:minrep} from above,
$\tilde{D} = \pi^{-1}(D)$ and with the convention that $1/p=0$ for $p=\infty$.

\medskip
If $\varphi = \sum_{|I|=q} \varphi_I d\bar{z}_I$ is a $(0,q)$-form that is not necessarily written in the minimal form
above, then we can make at least the following useful observation. Note that
$$|\varphi| \lesssim \sum_{|I|=q} |\varphi_I|,$$
and so
$$|\varphi|^p \lesssim \sum_{|I|=q} |\varphi_I|^p.$$
But
\begin{equation*}
    \pi^*(|\varphi_I|^p dV_X) = |\tilde{\varphi}_I|^p 2(\|w_1\|^4 + \|w_2\|^4 + 4\|w_1 w_2\|^2) dV(w).
\end{equation*}
So
\begin{equation*}
    \pi^*(|\varphi|^p dV_X) \leq C \xi^4 \sum |\tilde{\varphi}_I|^p dV(w),
\end{equation*}
and we have proved the first part of the following lemma.

\begin{lma}\label{lem:lpbound}
    Let $\varphi = \sum_{|I|=q} \varphi_I d\bar{z}_I$ be an arbitrary representation of
    $\varphi$ as a $(0,q)$-form on $X$. 
    \begin{itemize}
    \item[i.] If $\tilde{\varphi}_I \in \xi^{-4/p}L^p(\tilde{D})$ for all $I$, then $\varphi \in L^p_{0,q}(D)$.
    \item[ii.] If $\tilde{\varphi}_I$ is continuous at $p\in \tilde{D}$ for all $I$, then $\varphi$ is continuous at $\pi(p)\in D$.
    \end{itemize}
\end{lma}

\begin{proof}
It only remains to prove part ii. But continuity of $\tilde{\varphi}_I$ at $p$ implies directly continuity of $\varphi_I$ at $\pi(p)$
since $\pi$ is proper, and so $\varphi$ is continuous by definition.
\end{proof}

\smallskip
\subsection{Estimating integrals on the variety by estimates for the covering}

Using Lemma \ref{lem:lpbound},
we can now formulate a condition for an integral kernel on the variety to map $L^p$ into $L^p$
in terms of how the kernel behaves in the covering.
At the same time, we get some conditions on the convergence and boundedness, respectively,
of certain cut-off procedures to be studied later.

\begin{lma} \label{lmal2covering}
    Let $\mathcal{K}$ be an integral operator, acting on $(0,q)$-forms in $\zeta$ on $X'$,
    and returning $(0,q-1)$-forms in $z$ on $X$, and write the integral kernel $K$ in the form
    \begin{equation}\label{eq:integralkernels}
        K = \sum K_i \wedge d\bar{z}_i \text{ or } K = \sum K_i \wedge d\bar{\zeta}_i,
    \end{equation}
    depending on whether $q=2$ or $q=1$. Let $\tilde{K}_i = \pi^*K_i \wedge d\bar{w_1} \wedge d\bar{w_2}$.
    
    \medskip
    (i) Let $\frac{4}{3} < p \leq \infty$.
    If $\tilde{K}_i$ maps $\xi^{2-4/p}L^p(\tilde{D}')$ continuously to
    $\xi^{-4/p} L^p(\tilde{D})$ for $i=1,2,3$, then $\mathcal{K}$ maps $L^p_{0,q}(X')$ continuously to $L^p_{0,q-1}(X)$.
    
    \medskip
    (ii) If $\tilde{K}_i$ maps $\xi^{2} L^\infty(\tilde{D}')$ to functions continuous at $0\in\C^2$
    for $i=1,2,3$, then $\mathcal{K}$ maps $L^\infty_{0,q}(X')$ to functions continuous at $0\in X$.

    \medskip
    (iii) For $k\in\Z$, $k\geq 1$, let ${X}_k$ be a series of subdomains in $X'$ and $\tilde{D}_k=\pi^{-1}(X_k)$
    the corresponding subdomains of $\tilde{D}'$. Let $\mathcal{K}^k$ be the integral operators defined by 
    integrating against the kernel $K$ over $X_k$, and $\tilde{\mathcal{K}}_i^k$ the integral operators defined 
    by integrating against $\tilde{K}_i$ over $\tilde{D}_k$.
    
    If
    $$\int_{\tilde{D}} \|x\|^4 \big|\tilde{\mathcal{K}}^k_i \varphi (x)\big|^2 dV(x) \longrightarrow 0\ \ \ \mbox{ as }\ \ \ k\rightarrow \infty,$$
    i.e., $\tilde{\mathcal{K}}^k_i\varphi \rightarrow 0$ in $\xi^{-2}L^2(\tilde{D})$,
    for any $\varphi\in L^2(\tilde{D})$ and $i=1, 2, 3$, then
    $\mathcal{K}^k \varphi \rightarrow 0$ in $L^2_{0,q-1}(X)$ for any $\varphi\in L^2_{0,q}(X')$.
    
    \medskip
    (iv) If there exists a constant $C>0$ such that
    $$\int_{\tilde{D}_k} \frac{\big|\tilde{\mathcal{K}}_i \varphi(x)\big|^2}{\log^2\|x\|} dV(x) < C \|\varphi\|^2_{L^2(\tilde{D}')}$$
    for all $\varphi\in L^2(\tilde{D}')$, $k\geq 1$ and $i=1, 2, 3$, then there exists a constant $C'>0$ such
    that
    $$\int_{X_k} \frac{\big| \mathcal{K} \varphi(z)\big|^2}{\|z\|^2 \log^2 \|z\|} dV(z) < C' \|\varphi\|^2_{L^2(X')}$$
    for all $\varphi\in L^2(\tilde{X}')$, $k\geq 1$.
\end{lma}

\begin{proof}
Let us first prove parts (i) and (ii).
    We consider a $\varphi \in L^p_{0,q}(X')$, and write it as in \eqref{eq:minrep} above in the form
    $\varphi = \sum_{|I|=q} \varphi_I d\bar{\zeta}_I$, where $\varphi_I \in L^p(X')$.
    Thus, $\tilde{\varphi_I} = \pi^*\varphi_I \in \xi^{-4/p} L^p(\tilde{D}')$.

    We first consider the case $q=1$. Then $\pi^*(K\wedge\varphi)$ consists of terms
    \begin{equation} \label{eqpullback1}
    \pi^* K_i \wedge \pi^*(d\bar{\zeta}_i \wedge \varphi_j \wedge d\bar{\zeta}_j).
    \end{equation}
    Now, $\pi^*(d\bar{\zeta}_i \wedge d\bar{\zeta}_j) = C f d\bar{w_1}\wedge d\bar{w_2}$,
    where $C$ is a constant and $f$ is one the functions $\bar{w}_1^2$,$\bar{w}_1\bar{w}_2$ or $\bar{w}_2^2$,
    so $|f| \lesssim \xi^2$. We thus get
    that the second term in \eqref{eqpullback1} is $d\bar{w_1}\wedge d\bar{w_2}$ times a
    function in $\xi^{2-4/p}L^p(\tilde{D}')$.

    Thus, $\mathcal{K}$ acting on $\varphi$ expressed as an integral on $\tilde{D}'$
    will be of the form $\int_{\tilde{D}'} \tilde{K}_i \wedge d\bar{w_1}\wedge d\bar{w_2} \wedge \psi$,
    where $\psi \in \xi^{2-4/p}L^p(\tilde{D}')$. Thus, by assumptions on $\mathcal{K}$, 
    $\pi^* \mathcal{K}\varphi \in \xi^{-4/p} L^p(\tilde{D})$
    in case (i) and     
    $\pi^* \mathcal{K}\varphi$ is continuous at $0\in\C^2$ in case (ii).
    So, by Lemma \ref{lem:lpbound}, $\mathcal{K}\varphi \in L^p(X)$ (and this mapping is bounded) in case (i) and $\mathcal{K}\varphi$ 
    is continuous at $0\in X$ in case (ii).

\medskip
    In the same way, when $\varphi$ is a $(0,2)$-form, then $\pi^*\varphi$ will be a function in $\xi^{2-4/p}L^p(\tilde{D}')$
    times $d\bar{w_1}\wedge d\bar{w_2}$, so we can write $\mathcal{K} \varphi$ on the form $\mathcal{K}\varphi = \sum g_i d\bar{z}_i$,
    where $\pi^* g_i$ is of the form
    \begin{equation*}
        \int_{\tilde{D}'} \tilde{K}_i \wedge d\bar{w_1}\wedge d\bar{w_2} \wedge \psi
    \end{equation*}
    and just as above, we get that $\pi^* g_i \in \xi^{-4/p} L^p(\tilde{D})$ in case (i)
    and $\pi^* g_i$ is continuous at $0\in \C^2$ in case (ii).
    So, $g_i \in L^p(X)$, and thus,
    $\mathcal{K} \varphi \in L^p_{0,1}(X)$ in case (i).
    Analogously, $g_i$, and hence also $\mathcal{K}\varphi$, are continuous at $0\in X$ in case (ii).
        
    \medskip
    The proof of part (iii) and (iv) follows by exactly the same arguments (with $p=2$).
    For part (iv) recall that $\pi^* \|z\|^2 \sim \|x\|^4=\xi^4$.
\end{proof}

\bigskip

\section{Properties of the Andersson--Samuelsson integral operator at the $A_1$-singularity}\label{sec:main1}

\smallskip
\subsection{The Koppelman integral operator for a reduced complete intersection}

For convenience of the reader, let us recall shortly the definition of the Koppelman integral operators from \cite{AS2}
in the situation of a reduced complete intersection defined on two different open sets $D \subset\subset D' \subset\subset \C^N$,
$$X = \{ \zeta \in D \subset \C^N \mid g_1(\zeta) = \dots = g_p(\zeta) = 0 \}$$
and $$X' = \{ \zeta \in D' \subset \C^N \mid g_1(\zeta) = \dots = g_p(\zeta) = 0 \},$$
both of dimension $n=N-p$ (see \cite{AS2}, Section 8). 
Let $\omega_{X'}$ be a structure form on $X'$ (see \cite{AS2}, Section 3).
For generic coordinates $(\zeta',\zeta'')=(\zeta'_1, ..., \zeta'_p, \zeta''_1, ..., \zeta''_n)$
such that $\det\big( \partial g/\partial \zeta'\big)$ is generically non-vanishing on $X_{reg}'$,
the structure form $\omega_{X'}$ is essentially the pull-back of
\begin{eqnarray*}
\frac{d\zeta''_1 \wedge ... \wedge d\zeta''_n}{\det \big( \partial g / \partial \zeta')}
\end{eqnarray*}
to $X'$ (there are also some scalar constants and a fixed frame of a trivial line bundle).
The Koppelman integral operator $\mathcal{K}$, which is a homotopy operator for the $\dbar$-equation'
on $X$ is of the form,
\begin{equation}\label{eq:AS1}
    (\mathcal{K} \alpha)(z) = \int_{X'} K(\zeta,z) \wedge \alpha(\zeta),
\end{equation}
which takes forms on $X'$ as its input, and outputs forms on $X$.
Here,
\begin{equation*}
    K(\zeta,z) = \omega_{X'}(\zeta) \wedge \tilde{K}(\zeta,z),
\end{equation*}
and $\tilde{K}$ is defined by
\begin{equation*}
    \tilde{K}(\zeta,z) \wedge d\eta = h_1\wedge\dots\wedge h_p \wedge (g\wedge B)_n.
\end{equation*}
The \emph{Hefer forms} $h_i$ are $(1,0)$-forms satisfying $\delta_\eta h_i = g_i(\zeta) - g_i(z)$
where $\delta_\eta$ is the interior multiplication with
$$2\pi i \sum \eta_j \frac{\partial}{\partial \eta_j} = 2\pi i \sum (\zeta_j - z_j) \frac{\partial}{\partial \eta_j}.$$
The form $g$ is a so-called weight with compact support, and in case $D$ is the unit ball $D = B_1(0) \subseteq \C^n$,
then one choice of such a weight is
\begin{equation*}
    g = \chi - \dbar \chi \wedge \big(\sigma + \sigma(\dbar\sigma) + \dots + \sigma(\dbar\sigma)^{n-1}\big),
\end{equation*}
where
\begin{equation*}
    \sigma = \frac{\zeta \bullet d\eta}{2\pi i(\|\zeta\|^2-\bar{\zeta}\bullet z)}
\end{equation*}
and $\chi = \chi(\zeta)$ is a cut-off function which is identically $1$ in a neighborhood of $\bar{D}$, and
has support in $D'$.
The Bochner-Martinelli form $B$ is defined by
\begin{equation*}
    B = s + s\dbar s + \dots + s(\dbar s)^{n-1},
\end{equation*}
where
\begin{equation*}
    s = \frac{\partial \|\eta\|^2}{\|\eta\|^2} = \frac{\bar{\eta}\bullet d\eta}{\|\eta\|^2}.
\end{equation*}

Considering now the specific case when $X$ is the $A_1$-singularity, $X = \{ \zeta \in D \mid g(\zeta) = 0 \}$,
where $g(\zeta) = \zeta_1 \zeta_2 - \zeta_3^2$, then we choose as a Hefer form
\begin{equation*}
    h = \sum h^i d\eta_i = \frac{1}{2}\left( (\zeta_2 + z_2) d\eta_1 + (\zeta_1 + z_1) d\eta_2\right)
    - (\zeta_3 + z_3)d\eta_3,
\end{equation*}
and one representation of the structure form $\omega_{X'}$ is
\begin{equation}\label{eq:omegaX1}
    \omega_{X'} = \frac{d\zeta_1\wedge d\zeta_2}{-2\zeta_3}.
\end{equation}

\smallskip
\subsection{Proof of Theorem \ref{thm:main1}}\label{sec:main1b}

Note that
\begin{equation}\label{eq:omegaX2}
    \pi^* \omega_{X'} = (-1/2) dw_1\wedge dw_2
\end{equation}
under the $2$-sheeted covering $\pi: \C^2 \rightarrow X'$.
We then get that
\begin{equation*}
    \tilde{K} = \sum_{\sigma\in S_3}
    \frac{\chi}{\|\eta\|^4} h^{\sigma(1)} \bar{\eta}_{\sigma(2)} d\bar{\eta}_{\sigma(3)} -
    \frac{\dbar\chi}{2\pi i \|\eta\|^2 (\|\zeta\|^2-\bar{\zeta}\bullet z)} h^{\sigma(1)} \bar{\eta}_{\sigma(2)} \bar{\zeta}_{\sigma(3)},
\end{equation*}
where $S_l$ is the symmetric group on $l$ elements.
We decompose $\tilde{K}$ into $\tilde{K}_1$ and $\tilde{K}_2$, where $\tilde{K}_1$ and $\tilde{K}_2$
consist of the terms of $\tilde{K}$ containing $\chi$ and $\dbar\chi$, respectively.
The terms of $\tilde{K}_1$ and $\tilde{K}_2$ are then of the forms
\begin{equation*}
    \frac{g_1 f_1}{\|\eta\|^3} \wedge d\bar{\eta}_i \text{ and } \frac{g_2 f_2}{\|\eta\|(\|\zeta\|^2-\bar{\zeta}\bullet z)} \wedge \dbar\chi,
\end{equation*}
where $f_i$ is one of $\zeta_1,\zeta_2,\zeta_3,z_1,z_2,z_3$ and $g_i \in L^\infty(X'\times X)$
is a product of a smooth function with a term of the form $\eta_j/\|\eta\|$.
By Proposition \ref{prop:lpconvergence} below, it follows for $\pi^* g_i(w,x)=g_i(\pi(w),\pi(x))$ that
\begin{eqnarray}\label{eq:conv99}
\lim_{x\rightarrow 0} \pi^* g_i(\cdot,x) = \pi^*g_i(\cdot,0) \ \ \ \mbox{ in } L^r(\tilde{D}')
\end{eqnarray}
for all $1\leq r<\infty$.

The full kernel $K = \tilde{K} \wedge \omega_{X'}$ also splits into kernels
$K_i = \tilde{K_i} \wedge \omega_{X'}$.
We thus also get a decomposition $\mathcal{K} = \mathcal{K}_1 + \mathcal{K}_2$,
and we will prove separately that $\mathcal{K}_1$ and $\mathcal{K}_2$ 
have the claimed mapping properties.

If $\mathcal{K}_1$ is acting on $(0,1)$-forms or $(0,2)$-forms respectively, then
we will get a contribution from $K_1$ from terms of the form
\begin{equation*}
    \frac{g_1 f_1}{\| \eta \|^3}\wedge \omega_{X'} \wedge d\bar{\zeta}_i \text { or }
    \frac{g_1 f_1}{\| \eta \|^3}\wedge \omega_{X'} \wedge d\bar{z}_i
\end{equation*}
respectively.

Thus, by Lemma~\ref{lmal2covering}, (i), $\mathcal{K}_1$ maps $L^p_{0,q}(X')$ continuously to $L^p_{0,q-1}(X)$
if
\begin{equation} \label{eqk1coveringexp}
    \pi^*\left( \frac{g_1 f_1}{\| \eta \|^3} \omega_X \right) \wedge d\bar{w_1}\wedge d\bar{w_2} =
    c \frac{\tilde{g}_1 \tilde{f}_1}{\alpha^3} \wedge dV(w)
\end{equation}
maps $\xi^{2-4/p}L^p(\tilde{D}')$ continuously to $\xi^{-4/p}L^p(\tilde{D})$. 
But by Lemma~\ref{lmacoveringintegr},
a kernel of the form of \eqref{eqk1coveringexp} does indeed map $\xi^{2-4/p}L^p(\tilde{D}')$ continuously to $\xi^{-4/p} L^p(\tilde{D})$.

By the same Lemmata (and using \ref{eq:conv99}), 
$\mathcal{K}_1$ maps $L^\infty_{0,q}(X')$ to forms continuous at the origin $0\in X$.
On the other hand, on the regular part of $X$,
the kernel behaves like $\|\zeta- z\|^3$ in $\C^2$, i.e., like the Bochner-Martinelli-Koppelman kernel
(cf. also the proof of \cite{AS2}, Lemma 6.1).
So, $\mathcal{K}_1$ maps $L^\infty_{0,q}(X')$ to forms that are (H\"older-)continuous on $X\setminus\{0\}$
by standard arguments (see \cite{Range}, Theorem IV.1.14).
Summing up, we see that $\mathcal{K}_1$ maps $L^\infty_{0,q}(X')$ to $C^0_{0,q-1}(X)$.
This operator is continuous because the Fr\'echet space structure of $C^0_{0,q-1}(X)$
is defined by semi-norms $\|\cdot\|_{L^\infty,K_j}$ where $\{K_j\}_j$ is a compact exhaustion of $X$
(and $\mathcal{K}_1$ maps continuously from $L^\infty$ to $L^\infty$).

\medskip
Considering now $\mathcal{K}_2$, we note that since $\chi$ depends only on $\zeta$, the action
of $\mathcal{K}_2$ on $(0,2)$-forms is $0$, so we only need to consider the case of $(0,1)$-forms.
Note that we can write the pullback of the kernel acting on $\varphi_i d\bar{z_i}$ as an integral
on $X$ of the form
\begin{equation} \label{eqk2coveringexp}
    \pi^*\left( \frac{g \chi'(\zeta)}{\|\eta\|(\|\zeta\|^2-\bar{\zeta}\bullet z)}\right) dV(w)
\end{equation}
where $g \in L^{\infty}(\tilde{D}'\times\tilde{D})$ satisfies \eqref{eq:conv99}. 
Note that $\chi \equiv 1$ in a neighborhood of $\bar{X}$, so $\supp \chi' \cap \bar{X} = \emptyset \}$,
so the integrand in \eqref{eqk2coveringexp} is uniformly bounded
when $z \in X$ and $\zeta \in X'$, so the pullback of the kernel of $\mathcal{K}_2$ will define
bounded operator mapping $\xi^{2-4/p}L^p(\tilde{D}')$ to $\xi^{-4/p} L^p(\tilde{D})$.
By the same arguments as above, one gets also that $\mathcal{K}_2$ maps continuously from $L^\infty_{0,1}(X')$
to $C^0(X)$.

\bigskip
To complete the proof of Theorem \ref{thm:main1},
it only remains to prove the following:

\begin{prop}\label{prop:lpconvergence}
Let $a(\zeta,z) = \frac{\zeta_i-z_i}{\|\zeta-z\|}$. Then
\begin{eqnarray}\label{eq:est99}
\lim_{z\rightarrow 0} a(\cdot,z) = a(\cdot,0) \ \ \ \mbox{ in } L^r(X')
\end{eqnarray}
for all $1\leq r < \infty$. Let $\pi^* a(w,x) := a(\pi(w),\pi(x))$. Then
\begin{eqnarray}\label{eq:est99b}
\lim_{x\rightarrow 0} \pi^* a(\cdot,x) = \pi^*a(\cdot,0) \ \ \ \mbox{ in } L^s(\tilde{D}')
\end{eqnarray}
for all $1\leq s < \infty$.
\end{prop}

\begin{proof}
Fix $1\leq r<\infty$ and note that $a$ is bounded (by $\|a\|_\infty=1$).

Let $0\leq \chi\leq 1$ be a smooth function such that $\chi\equiv 0$ on $B_{1/2}(0)$
and $\chi\equiv 1$ on $\C^3\setminus B_{1}(0)$, and set $\chi_\epsilon(x) := \chi(x/\epsilon)$ (let $0<\epsilon<1$ throughout this proof).
Let
$$a_\epsilon(\zeta,z) := \chi_\epsilon(\zeta-z) a(\zeta,z).$$
Then $a_\epsilon(\zeta,z)$ is smooth and it is not hard to see by Lebesgue's theorem on dominated
convergence that
\begin{eqnarray*}
\lim_{\epsilon\rightarrow 0} a_\epsilon(\cdot,z) = a(\cdot,z) \ \ \ \mbox{ in } L^r(X')
\end{eqnarray*}
for all $z\in \C^3$. We can say more, namely this convergence is uniformly in $z$:
\begin{eqnarray*}
\|a_\epsilon(\cdot,z) - a(\cdot,z) \|_{L^r(X')} &=& \|a_\epsilon(\cdot,z) - a(\cdot,z)\|_{L^r(X'\cap B_\epsilon(z))}\\
&\leq& \|a\|_\infty \|\chi_\epsilon(\cdot-z)-1\|_{L^r(X'\cap B_\epsilon(z))}\\
&\leq& \|a\|_\infty \int_{X'\cap B_\epsilon(z)} dV_{X'} \lesssim \epsilon^4,
\end{eqnarray*}
because $X'$ is a complex variety of dimension $2$. This follows from \cite{Dem}, Consequence III.5.8,
because $X'$ is bounded and has Lelong number $\leq 2$.

\medskip
We can now prove \eqref{eq:est99}.
Let $\delta>0$. 
By the considerations above, we can choose $\epsilon>0$ such that
\begin{eqnarray*}\label{eq:est100}
\|a_\epsilon(\cdot,z) - a(\cdot,z) \|_{L^r(X')} &\leq& \delta/3
\end{eqnarray*}
for all $z\in \C^3$. Fix such an $\epsilon>0$.
It follows that
\begin{eqnarray*}
\|a(\cdot,z) - a(\cdot,0)\|_{L^r(X')} &\leq& 
\|a(\cdot,z) - a_\epsilon(\cdot,z)\|_{L^r(X')}
+ \|a_\epsilon(\cdot,z) - a_\epsilon(\cdot,0)\|_{L^r(X')}\\
&& + \|a_\epsilon(\cdot,0) - a(\cdot,0)\|_{L^r(X')}\\
&\leq& 2\delta/3 + \|a_\epsilon(\cdot,z) - a_\epsilon(\cdot,0)\|_{L^r(X')}
\end{eqnarray*}
for all $z\in\C^3$.
On the other hand, $a_\epsilon$ is smooth on $\C^3\times\C^3$,
and so there exists a constant $C>0$ such that
$$|a_\epsilon(\zeta,z) - a_\epsilon(\zeta,0)| \leq C \|z\|$$
for all $\zeta,z$ in a bounded domain. Hence, we get that
$$\|a_\epsilon(\cdot,z) - a_\epsilon(\cdot,0)\|_{L^r(X')} \leq \delta/3$$
if $\|z\|$ is small enough.

Summing up, we have found that actually
\begin{eqnarray*}
\|a(\cdot,z) - a(\cdot,0)\|_{L^r(X')} &\leq& \delta
\end{eqnarray*}
if $\|z\|$ is small enough. 

\medskip
That proves the first statement of the proposition.
For the second part, fix $1\leq s <\infty$.
Recall from Section \ref{sec:lp-forms}
that, for functions, the $L^r$-norm on $X'$ is equivalent to the $\|w\|^{-4/r} L^r$-norm on $\tilde{D}'$.
But, by the H\"older-inequality, convergence in $\|w\|^{-4/r}L^r$ implies convergence in $L^s$
if $r<\infty$ is chosen large enough. So, the second statement follows from the first one
if we just choose $1\leq r<\infty$ large enough (depending on $1\leq s<\infty$).
\end{proof}

\bigskip

\section{The $L^p$-homotopy formula for the $\dbar$-operator in the sense of distributions}\label{sec:main2}

The original $\dbar$-homotopy formula of Andersson--Samuelsson holds only for forms
on the variety $X$ which are the restriction of smooth forms on a neighborhood of the variety
(or, more generally, for forms with values in the $\mathcal{A}$-sheaves mentioned in the introduction;
see \cite{AS2}, Theorem 1.4). So, in order to extend the $\dbar$-homotopy formula to $L^p$-forms
given only on the variety, we need to approximate these in an appropriate way by smooth forms
extending to a neighborhood of the variety. To do so, we need to cut-off the forms so that they
vanish in neighborhoods of the singularity.

\smallskip
\subsection{Estimates for the cut-off procedure}\label{sec:main2b}

We will use the following cut-off functions to approximate forms
by forms with support away from the singularity in different situations.

As in \cite{PS}, Lemma 3.6, let $\rho_k: \R\rightarrow [0,1]$, $k\geq 1$, be smooth cut-off functions
satisfying 
$$\rho_k(x)=\left\{\begin{array}{ll}
1 &,\  x\leq k,\\
0 &,\  x\geq k+1,
\end{array}\right.$$
and $|\rho_k'|\leq 2$. Moreover, let $r: \R\rightarrow [0,1/2]$ be a smooth increasing function such that
$$r(x)=\left\{\begin{array}{ll}
x &,\ x\leq 1/4,\\
1/2 &,\ x\geq 3/4,
\end{array}\right.$$
and $|r'|\leq 1$.

As cut-off functions we can use 
\begin{eqnarray}\label{eq:cutoff1}
\mu_k(\zeta):=\rho_k\big(\log(-\log r(\|\zeta\|))\big)
\end{eqnarray} 
on $X$. Note that
\begin{eqnarray}\label{eq:cutoff2}
\big| \dbar \mu_k(\zeta)\big| \lesssim \frac{\chi_k(\|\zeta\|)}{\|\zeta\| \big| \log\|\zeta\|\big|},
\end{eqnarray}
where $\chi_k$ is the characteristic function of $[e^{-e^{k+1}}, e^{-e^k}]$.

\begin{thm}\label{thm:main2}
    Let $\mathcal{K}$ be integral operator from Theorem \ref{thm:main1},
    and let $\varphi\in L^2_{0,q}(X')$, $1\leq q \leq 2$. Then
\begin{eqnarray*}
\mathcal{K} \big( \dbar\mu_k \wedge \varphi \big) &\longrightarrow& 0
\end{eqnarray*}
in $L^2_{0,q}(X')$ as $k\rightarrow \infty$.
\end{thm}

\begin{proof}
By \eqref{eq:cutoff2}, we see that 
$$\big|\mathcal{K}\big( \dbar\mu_k \wedge \varphi\big)\big| \lesssim 
\big|\mathcal{K}\big|\left( \frac{\chi_k(\|\zeta\|) |\varphi|}{\|\zeta\| \big| \log\|\zeta\| \big|}\right),$$
where if $\mathcal{K}$ is the integral operator defined by the integral kernel $K(\zeta,z)$,
then $|\mathcal{K}|$ is the integral operator defined by the integral kernel $|K(\zeta,z)|$.
So, let
$$\mathcal{K}^k \varphi := \big|\mathcal{K}\big|\left( \frac{\chi_k(\|\zeta\|) \varphi}{\|\zeta\| \big| \log\|\zeta\| \big|}\right)\ \ ,k\geq 1,$$
be the corresponding series of integral operators on $X_k:= X'\cap \supp \chi_k$.

Proceeding as in the proof of Theorem \ref{thm:main1} (let $g_1$, $f_1$, $\tilde{g}_1$, $\tilde{f}_1$
be as in \eqref{eqk1coveringexp}),
we see by Lemma \ref{lmal2covering}, (iii), that actually
$$\mathcal{K}^k \varphi \longrightarrow 0$$
in $L^2_*(X')$ if the kernels
\begin{align} \label{eqk1coveringexp2}
    \tilde{K}_k := \left| \pi^*\left( \frac{\chi_k(\|\zeta\|)}{\|\zeta\| \big|\log\|\zeta\|\big|} 
    \frac{g_1 f_1}{\| \eta \|^3} \omega_X \right) \wedge d\bar{w_1}\wedge d\bar{w_2} \right|
\end{align}
define a series of integral operators $\tilde{\mathcal{K}}^k$, $k\geq 1$ such that
\begin{equation} \label{eq:ktildezero}
    \tilde{\mathcal{K}}^k \varphi \longrightarrow 0
\end{equation}
in $\xi^{-2} L^2(\tilde{D})$ for any $\varphi\in L^2(\tilde{D}')$.
Since $(1/2) \|w\|^2 \leq \pi^* \|\zeta\| \leq \|w\|^2$, we get that
$\pi^* \chi_k(\|\zeta\|) \leq \chi_{\tilde{D}_k}(w)$,
where $\chi_{\tilde{D}_k}$ is the characteristic function on $\tilde{D}_k$ as given by \eqref{eq:Dk},
and we then also get that
\begin{align*}
    \tilde{K}_k \lesssim
    \frac{\chi_{\tilde{D}_k}(w)}{\|w\|^2 \big|\log\|w\|\big|} \frac{\left| \tilde{g}_1 \tilde{f}_1 \right|}{\alpha^3} \wedge dV(w).
\end{align*}
Thus, we conclude that \eqref{eq:ktildezero} holds by Lemma \ref{lem:cutoff2}.
\end{proof}

\smallskip
\subsection{Proof of Theorem \ref{thm:main3}}

In order to apply the $\dbar$-homotopy formulas of Andersson-Samuelsson to $\varphi$
we need to approximate $\varphi$ and its $\dbar$-derivative by smooth forms on a neighborhood of $X$.
This can be done appropriately by use of the cut-off functions introduced in Section \ref{sec:main2b}.
So, let 
$$\phi_k := \mu_k \varphi,$$
where $\mu_k$ is the cut-off sequence from Section \ref{sec:main2b}.
By Lebesgue's theorem on dominated convergence, note that
\begin{align}\label{eq:cutoff01}
\phi_k \rightarrow \varphi\ \ , \ \ \mu_k \dbar \varphi \rightarrow \dbar\varphi\ \ \mbox{ in } L^p_{0,*}(X').
\end{align} 
As the $\phi_k$ have support away from the singular point, we can apply Friedrichs' density lemma:
just use a standard smoothing procedure, i.e., convolution with a Dirac sequence,
on the smooth manifold $X^*$ (cf., \cite{LiMi}*{Theorem~V.2.6}). 
So, there are sequences of smooth forms $\phi_{k,l}$ with support away from the singular point
such that 
\begin{align}\label{eq:cutoff01b}
\phi_{k,l} \overset{l\rightarrow \infty}{\longrightarrow} \phi_k\ \ , 
\ \ \dbar\phi_{k,l} \overset{l\rightarrow\infty}{\longrightarrow} \dbar\phi_k\ \ \mbox{ in } L^p_{0,*}(X').
\end{align} 
Now the $\phi_{k,l}$ can be extended smoothly
to a neighborhood of $X$ and it follows by the $\dbar$-homotopy formula of Andersson-Samuelsson, \cite{AS2}, Theorem 1.4,
that
\begin{eqnarray*}
\phi_{k,l} &=& \dbar \mathcal{K} \phi_{k,l} + \mathcal{K} \dbar \phi_{k,l}
\end{eqnarray*}
in the sense of distributions on $X$ for all $k,l\geq 1$. From this,
it follows by \eqref{eq:cutoff01b} and Theorem \ref{thm:main1} (letting $l\rightarrow \infty$)
that the homotopy formula holds for all $\phi_k$, $k\geq 1$:
\begin{eqnarray*}
\phi_k &=& \dbar \mathcal{K} \phi_k + \mathcal{K} \dbar \phi_k \\
&=& \dbar \mathcal{K} \phi_k + \mathcal{K} \big(\mu_k \dbar \varphi\big) + \mathcal{K} \big( \dbar \mu_k \wedge \varphi\big)
\end{eqnarray*}
in the sense of distributions on $X$ for all $k\geq 1$.

\medskip
Using \eqref{eq:cutoff01} and Theorem \ref{thm:main1} again, we see that
\begin{eqnarray}\label{eq:cutoff02}
\mathcal{K} \phi_k \rightarrow \mathcal{K}\varphi \ \ \mbox{ and } 
\ \ \mathcal{K}\big(\mu_k\dbar\varphi\big) \rightarrow \mathcal{K} \dbar \varphi
\end{eqnarray}
in $L^p_*(X)$. Moreover, using $L^p_*(X') \subset L^2_*(X')$ and Theorem \ref{thm:main2},
we also get that
\begin{eqnarray}\label{eq:cutoff03}
\mathcal{K}\big(\dbar \mu_k\wedge \varphi\big) \rightarrow 0
\end{eqnarray}
in $L^2_{0,q}(X)$. So, it follows that actually
$\varphi = \dbar \mathcal{K}\varphi + \mathcal{K}\big( \dbar \varphi\big)$
in the sense of distributions on $X$.

\bigskip

\section{Other variants of the $\dbar$-operator} \label{sec:main3}

\smallskip
\subsection{The strong $\dbar$-operator $\dbar_s$}

In this section, we give the proof of Theorem \ref{thm:main4}.

Note first that in order to prove that
$\phi \in \Dom \dbar_s \subset L^2_{0,q}(X)$, it is sufficient to
find a sequence $\{ \phi_j \}_j \subset \Dom \dbar_w \subset L^2_{0,q}(X)$
with $\esssupp \phi_j \cap \{ 0 \} = \emptyset$ such that
\begin{eqnarray}
\phi_j \rightarrow \phi \ \ \ &\mbox{ in }& \ \ L^2_{0,q}(X),\\
\dbar \phi_j \rightarrow \dbar \phi \ \ \ &\mbox{ in }& \ \ L^2_{0,q+1}(X),
\end{eqnarray}
i.e., it is not necessary to assume that the $\phi_j$ are smooth, since
if we assume that the $\phi_j$'s have support outside of the singular set of $X$,
then by Friedrichs' extension lemma, \cite{LiMi}*{Theorem~V.2.6}, on the complex manifold $X^*$,
there exists smooth $\tilde{\phi}_j \in L^2_{0,q}(X)$ with support away from $\{0\}$ such that $\|\phi_j-\tilde{\phi}_j\|_{L^2}$
and $\|\dbar \phi_j-\dbar\tilde{\phi}_j\|_{L^2}$ are arbitrarily small.

So, let $\varphi\in \Dom \dbar_w \subseteq L^2_{0,q}(X')$, where $1\leq q \leq 2$.
Let $\mu_k$ be the cut-off sequence from Section \ref{sec:main2b} and set
$$\phi_k := \mu_k \mathcal{K} \varphi.$$
Then $\{\phi_k\}_k \subset L^2_{0,q-1}(X)$ and it follows by Lebesgue's theorem on dominated convergence
that
\begin{eqnarray*}
\phi_k \rightarrow \mathcal{K}\varphi \ \ \ &\mbox{ in }& \ \ L^2_{0,q-1}(X),\\
\dbar \phi_k -\dbar\mu_k\wedge \mathcal{K}\varphi = 
\mu_k \dbar \mathcal{K}\varphi \rightarrow \dbar \mathcal{K} \varphi \ \ \ &\mbox{ in }& \ \ L^2_{0,q}(X)
\end{eqnarray*}
as $k\rightarrow \infty$ since $\mathcal{K} \varphi \in L^2_{0,q-1}(X)$ by Theorem~\ref{thm:main1},
and thus, $\dbar \mathcal{K} \varphi \in L^2_{0,q}(X)$ by Theorem~\ref{thm:main3}.
To see that actually 
\begin{eqnarray}\label{eq:claim}
\mathcal{K} \varphi &\in& \Dom\dbar_s \subset L^2_{0,q-1}(X),
\end{eqnarray}
we claim that it is enough to show that the set of forms
\begin{eqnarray}\label{eq:claim2}
\big\{\dbar \mu_k \wedge \mathcal{K}\varphi \big\}_k
\end{eqnarray}
is uniformly bounded in $L^2_{0,q}(X)$. This can be seen by the following duality argument:

\medskip
{\bf Proof of the claim:}
Let \eqref{eq:claim2} be uniformly bounded in $L^2_{0,q}(X)$, independent of $k$.
We can assume that $\mathcal{K}\varphi$ has compact support in a small neighborhood, say $V$, of the origin.
Then, refeering to the notation in \cite{RuDuke}, Section 2.4,
we need to show that $\mathcal{K}\varphi\in\Dom \dbar_{min}$.

But, on the Hermitian manifold $X\cap V\setminus \{0\}$,
the $L^2$-adjoint operator of $\dbar_{min}$ is $\vartheta_{max}$ (see \cite{RuDuke}, Section 2.4).
So, to show the claim, we have to prove that
\begin{eqnarray}\label{eq:claim3}
\left ( \mathcal{K}\varphi, \vartheta_{max} g\right)_{L^2(X)} &=& \left( \dbar \mathcal{K}\varphi, g\right)_{L^2(X)}
\end{eqnarray}
for all $g\in \Dom\vartheta_{max} \subset L^2_{0,q+1}(X\cap V)$.
For such a $g$, we compute:
\begin{eqnarray*}
\left ( \mathcal{K}\varphi, \vartheta_{max} g\right)_{L^2(X)} &=&
\lim_{k\rightarrow\infty} \left ( \phi_k, \vartheta_{max} g\right)_{L^2(X)}\\
&=& \lim_{k\rightarrow\infty} \left ( \dbar \phi_k, g\right)_{L^2(X)}\\
&=&  \left( \dbar \mathcal{K}\varphi, g\right)_{L^2(X)}
+ \lim_{k\rightarrow\infty} \left ( \dbar \mu_k \wedge \mathcal{K}\varphi, g\right)_{L^2(X)}.
\end{eqnarray*}
But, as \eqref{eq:claim2} is uniformly bounded, we have furthermore:
\begin{eqnarray*}
\left|\left ( \dbar \mu_k \wedge \mathcal{K}\varphi, g\right)_{L^2(X)}\right|
&\lesssim& \|g\|_{L^2(\supp \dbar\mu_k)} \overset{k\rightarrow \infty}{\longrightarrow} 0,
\end{eqnarray*}
because $g$ is square-integrable and the domain of integration vanishes.
This proves the claim.

\qed

\bigskip
To show that \eqref{eq:claim2} is uniformly bounded,
we proceed similarly as in the proof of Theorem \ref{thm:main2}.
By \eqref{eq:cutoff2}, we see that 
$$\left| \dbar \mu_k\wedge \mathcal{K}\varphi \big(z\big) \right| \lesssim  \frac{\chi_k(\|z\|)}{\|z\| \big| \log\|z\| \big|} \wedge \left| \mathcal{K}\varphi (z) \right| .$$
So, let
$$\mathcal{K}^k \varphi (z):=  \frac{\chi_k(\|z\|)}{\|z\| \big| \log\|z\| \big|}\wedge\mathcal{K}\varphi(z) \ \ ,k\geq 1,$$
be the corresponding series of integral operators on $X'$.

Proceeding as in the proof of Theorem \ref{thm:main1} (let $g_1$, $f_1$, $\tilde{g}_1$, $\tilde{f}_1$
be as in \eqref{eqk1coveringexp}),
we see by Lemma \ref{lmal2covering}, (iv), that actually
$$\{\mathcal{K}^k \varphi\}_k$$
is uniformly bounded in $L^2_*(X)$ if the kernels
\begin{align} \label{eqk1coveringexp3}
    \tilde{K}_k := \left| \pi^*\left( \frac{\chi_k(\|z\|)}{\|z\| \big|\log\|z\|\big|} 
    \frac{g_1 f_1}{\| \eta \|^3} \omega_X \right) \wedge d\bar{w_1}\wedge d\bar{w_2} \right|
\end{align}
define a series of integral operators $\tilde{\mathcal{K}}^k$ on $\tilde{D}$ 
such that
\begin{equation}
    \label{eq:Ktilde}
    \{\tilde{\mathcal{K}}^k\varphi \}_k
\end{equation}
is uniformly bounded in $\xi^{-2} L^2(\tilde{D})$
for any $\varphi\in L^2(\tilde{D})$. As in the end of the proof of Theorem \ref{thm:main2},
we get that
\begin{equation*}
    \tilde{K}_k \lesssim \frac{\chi_{\tilde{D}_k}(x)}{\|x\|^2 \big|\log\|x\|\big|} \frac{| \tilde{g}_1 \tilde{f}_1 |}{\alpha^3} \wedge dV(w),
\end{equation*}
and thus, \eqref{eq:Ktilde} is uniformly bounded by Lemma \ref{lem:cutoff1}.

\medskip
\subsection{Andersson--Samuelsson's operator $\dbar_X$}

In this section, we give the proof of Theorem \ref{thm:main5}.

By Theorem \ref{thm:main4}, $\mathcal{K}\varphi \in \Dom\dbar_s$. So, there is a sequence $\{\psi_j\}_j$
of smooth forms with support away from the singular point, $\supp \psi_j \cap \{0\} =\emptyset$,
and such that 
\begin{eqnarray}\label{eq:conv}
\psi_j \rightarrow \mathcal{K}\varphi\ \ \mbox{ and }\ \ \dbar\psi_j \rightarrow \dbar \mathcal{K}\varphi
\end{eqnarray}
in the $L^2$-sense on $X$ as $j\rightarrow \infty$ (see \eqref{eq:dbars1}, \eqref{eq:dbars2}).

By \cite{AS2}, Proposition 1.5, $\mathcal{K}\varphi \in \mathcal{W}(X)$. In addition, since we assume that $\dbar \varphi \in L^2$,
we get by Theorem~\ref{thm:main1} that $\mathcal{K} \dbar \varphi \in L^2(X)$, and by Theorem~\ref{thm:main3}, we then get
that $\dbar \mathcal{K} \varphi \in L^2(X)$. Since $\mathcal{K} \varphi \in \mathcal{W}(X) \subseteq \PM(X)$, also 
$\dbar \mathcal{K} \varphi \in \PM(X)$ since $\PM(X)$ is closed under $\dbar$. Hence, $\dbar \mathcal{K} \varphi \in L^2(X) \cap \PM(X)$,
and by dominated convergence, we get that $\dbar \mathcal{K} \varphi \in \mathcal{W}(X)$.

We have to show that
\begin{eqnarray}\label{eq:dbarX1}
\dbar \big( \mathcal{K}\varphi\wedge \omega_X\big) = \big( \dbar\mathcal{K}\varphi\big) \wedge \omega_X
\end{eqnarray}
in the sense of distributions (see \cite{AS2}, Proposition 4.4). 
But $\omega_X \in L^2_{2,0}(X)$ by \eqref{eq:omegaX1} and Lemma \ref{lem:lpbound} (consider $\overline{\omega_X}$).
So, $\psi_j \rightarrow \mathcal{K}\varphi$ in $L^2_{0,q-1}(X)$ implies by use of the H\"older inequality
that $\psi_j\wedge \omega_X \rightarrow \mathcal{K}\varphi\wedge \omega_X$ in the sense of distributions.
By the same argument, we see that 
$\big(\dbar \psi_j\big)\wedge \omega_X \rightarrow \big(\dbar\mathcal{K}\varphi\big)\wedge \omega_X$ in the sense of distributions.
But $\psi_j\in \Dom \dbar_X$, i.e., $\dbar\big( \psi_j\wedge\omega_X\big) = \big(\dbar\psi_j \big)\wedge\omega_X$ in the sense of distributions.
So, we actually have
\begin{eqnarray*}
\dbar \big( \mathcal{K}\varphi\wedge \omega_X\big) = \lim_{j\rightarrow\infty} \dbar \big( \psi_j\wedge \omega_X\big)
= \lim_{j\rightarrow \infty} \big(\dbar\psi_j\big) \wedge\omega_X = \big(\dbar\mathcal{K}\varphi\big)\wedge\omega_X
\end{eqnarray*}
in the sense of distributions.

\bigskip

\appendix

\section{Estimates for integral kernels in $\C^n$}\label{sec:appendix}

\begin{lma}\label{lem:estimateCn1}
Let $\alpha\in \R$.
Then there exists a constant $C_\alpha>0$ such that the following holds:
\begin{align}\label{eq:estimateCn1}
I(r_1,r_2) := \int_{B_{r_2}(x)\setminus \overline{B_{r_1}(x)}} 
\frac{dV_{\C^n}(\zeta)}{\|\zeta-x\|^\alpha} \leq C_\alpha \left\{ 
\begin{array}{ll}
r_2^{2n-\alpha} & \ ,\ \alpha<2n,\\
|\log r_2| + |\log r_1|& \ ,\ \alpha=2n,\\
r_1^{2n-\alpha} & \ ,\ \alpha>2n,
\end{array}\right.
\end{align}
for all $x\in \C^n$ and all $0<r_1 \leq r_2< \infty$.
\end{lma}

\begin{proof}
A simple calculation, using Fubini, gives:
\begin{eqnarray*}
I(r_1,r_2) &=& \int_{r_1}^{r_2} \int_{bB_t(x)} \frac{dS_{bB_t(x)}(\zeta)}{t^\alpha} dt
\sim \int_{r_1}^{r_2} \frac{t^{2n-1}}{t^\alpha} dt \\
&\lesssim& \left\{
\begin{array}{ll}
r_2^{2n-\alpha} - r_1^{2n-\alpha} & , \alpha <2n\\
\log r_2 - \log r_1 & , \alpha=2n\\
r_1^{2n-\alpha} - r_2^{2n-\alpha} & , \alpha > 2n
\end{array}\right\}
\leq  \left\{
\begin{array}{ll}
r_2^{2n-\alpha}  & , \alpha <2n,\\
 |\log r_2| + |\log r_1| & , \alpha=2n,\\
r_1^{2n-\alpha} & , \alpha > 2n.
\end{array}\right.
\end{eqnarray*}
\end{proof}

From that we can deduce our first basic estimate:

\begin{lma}\label{lem:estimateCn2}
Let $D\subset \subset \C^n$ be a bounded domain and $0 \leq \alpha,\beta <2n$. 
Then there exists a constant $C_1>0$ such that the following holds:
\begin{align}\label{eq:estimateCn2}
\int_{D} \frac{dV_{\C^n}(\zeta)}{\|\zeta-x_1\|^\alpha \|\zeta-x_2\|^\beta}
\leq C_1 \left\{
\begin{array}{ll}
1 & \ ,\ \alpha+\beta<2n,\\
1+ \big|\log \|x_1 - x_2\|\big| &\ ,\ \alpha+\beta=2n,\\
\|x_1 - x_2\|^{2n-\alpha-\beta} &\ ,\ \alpha+\beta>2n,
\end{array}\right.
\end{align}
for all $x_1, x_2 \in \C^n$ with $x_1\neq x_2$.
\end{lma}

\begin{proof}
Let $R/2$ be the diameter of $D$ in $\C^n$.
We can assume that $D$ is not empty and that $R/2>0$.
Further, we can assume that $\dist_{\C^n}(D,x_1) < R/2$ (otherwise, the estimate just gets easier).
This implies
\begin{eqnarray}\label{eq:R1}
D \subset B_R(x_1).
\end{eqnarray}

Let $\delta:=\|x_1-x_2\|/3$.
We divide the domain of integration in three regions $D_1$, $D_2$ and $D\setminus(D_1\cup D_2)$. Let 
$$D_1:= D \cap B_{\delta}(x_1)\ \ ,\ \ D_2:= D \cap B_{\delta}(x_2)$$
Then $\|\zeta-x_2\| \geq \delta$ on $D_1$ and so
\begin{align}\label{eq:es11}
\int_{D_1} \frac{dV_{\C^n}(\zeta)}{\|\zeta-x_1\|^\alpha \|\zeta-x_2\|^\beta} \leq \delta^{-\beta}
\int_{B_{\delta}(x_1)} \frac{dV_{\C^n}(\zeta)}{\|\zeta-x_1\|^\alpha}
&\leq& C_\alpha \delta^{-\beta+2n-\alpha}
\end{align}
by use of Lemma \ref{lem:estimateCn1} (using $\alpha<2n$ and letting $r_2=\delta$, $r_1\rightarrow 0$).

As $\|\zeta-x_1\|\geq \delta$ on $D_2$, analogously:
\begin{eqnarray}\label{eq:es12}
\int_{D_2} \frac{dV_X(\zeta)}{\|\zeta-x_1\|^\alpha \|\zeta-x_2\|^\beta} 
&\leq& C_\beta \delta^{-\alpha+2n-\beta}
\end{eqnarray}

It remains to consider the integral over $D \setminus (D_1\cup D_2)$.
Here, $\|\zeta-x_2\|\geq \delta$ and that yields:
\begin{eqnarray*}
\|\zeta - x_1\| \leq \|\zeta-x_2\| + \|x_1-x_2\| = \|\zeta- x_2\| + 3 \delta \leq 4 \|\zeta-x_2\|.
\end{eqnarray*}
So, we can estimate by use of \eqref{eq:R1} and Lemma \ref{lem:estimateCn1}:
\begin{eqnarray*}
&& \int_{D \setminus (D_1\cup D_2)} \frac{dV_{\C^n} (\zeta)}{\|\zeta-x_1\|^\alpha \|\zeta-x_2\|^\beta}
 \leq 4^{\beta} \int_{B_{R}(x_1)\setminus \overline{B_{\delta}(x_1)}} \frac{dV_{\C^n}(\zeta)}{\|\zeta-x_1\|^{\alpha+\beta}}\\
&\leq& 4^{\beta} C_{\alpha+\beta}  \left\{ 
\begin{array}{ll}
R^{2n-\alpha-\beta} & \ ,\ \alpha+\beta<2n,\\
|\log R|+|\log \delta| & \ ,\ \alpha+\beta=2n,\\
\delta^{2n-\alpha-\beta} & \ ,\ \alpha+\beta>2n,
\end{array}\right.
\end{eqnarray*}
The assertion follows easily from this statement in combination with \eqref{eq:es11} and \eqref{eq:es12}.
\end{proof}

Another basic estimate is:

\begin{lma}\label{lem:estimateCn3}
Let $D\subset \subset \C^n$ be a bounded domain, $0 \leq \alpha,\beta <2n$ and $\gamma > -2n$. 
Then there exists a constant $C_2>0$ such that the following holds:
\begin{align}\label{eq:estimateCn3}
\int_{D} \frac{\|\zeta\|^\gamma dV_{\C^n}(\zeta)}{\|\zeta-x\|^\alpha \|\zeta+x\|^\beta}
\leq C_2 \left\{
\begin{array}{ll}
1 & \ ,\ \alpha+\beta<2n+\gamma,\\
1+ \big|\log \|x\|\big| &\ ,\ \alpha+\beta=2n+\gamma,\\
\|x\|^{2n+\gamma-\alpha-\beta} &\ ,\ \alpha+\beta>2n+\gamma,
\end{array}\right.
\end{align}
for all $x \in \C^n$ with $x\neq 0$.
\end{lma}

\begin{proof}
We can proceed similar as in the proof of Lemma \ref{lem:estimateCn2},
but have to divide $D$ into four domains. We can assume that $D$ is contained in a ball $B_R(0)$.

Let $\delta:=\|x\|/3$ and set
$$D_0:= B_\delta(0)\ \ ,\ \ D_1:= B_\delta(x)\ \ ,\ \ D_2:= B_\delta(-x).$$
Then $\|\zeta + x\|=\|(\zeta-x) + 2x\| \geq 5\delta$ and $\|\zeta\| \leq 4 \delta$ on $D_1$ and so we obtain
\begin{align}\label{eq:es13}
\int_{D_1} \frac{\|\zeta\|^\gamma dV_{\C^n}(\zeta)}{\|\zeta-x\|^\alpha \|\zeta+x\|^\beta}
\lesssim  \delta^{\gamma-\beta} \int_{B_\delta(x)} \frac{dV_{\C^n}(\zeta)}{\|\zeta-x\|^\alpha}
\leq C_\alpha  \delta^{2n+\gamma-\alpha-\beta}
\end{align}
by use of Lemma \ref{lem:estimateCn1}. Analogously,
\begin{align}\label{eq:es14}
\int_{D_2} \frac{\|\zeta\|^\gamma dV_{\C^n}(\zeta)}{\|\zeta-x\|^\alpha \|\zeta+x\|^\beta}
\lesssim  \delta^{\gamma-\alpha} \int_{B_\delta(-x)} \frac{dV_{\C^n}(\zeta)}{\|\zeta+x\|^\beta}
\leq C_\beta  \delta^{2n+\gamma-\alpha-\beta}.
\end{align}
Similarly, we have $\|\zeta-x\|\geq 2\delta$ and $\|\zeta+x\|\geq 2\delta$ on $D_0$
and that gives
\begin{align}\label{eq:es15}
\int_{D_0} \frac{\|\zeta\|^\gamma dV_{\C^n}(\zeta)}{\|\zeta-x\|^\alpha \|\zeta+x\|^\beta}
\leq  \delta^{-\alpha-\beta} \int_{B_\delta(0)} \|\zeta\|^\gamma dV_{\C^n}
\leq C_\gamma  \delta^{2n+\gamma-\alpha-\beta}.
\end{align}
Finally, we have to consider $D\setminus\big(D_0\cup D_1\cup D_2\big)$.
Here,
$$\|\zeta\| \leq \|\zeta-x\| + \|x\| = \|\zeta-x\| + 3\delta \leq 4 \|\zeta-x\|,$$
and analogously $\|\zeta\| \leq 4 \|\zeta+x\|$.
From that we deduce:
\begin{eqnarray*}
&& \int_{D\setminus\big(D_0\cup D_1\cup D_2\big)} \frac{\|\zeta\|^\gamma dV_{\C^n}(\zeta)}{\|\zeta-x\|^\alpha \|\zeta+x\|^\beta}
\leq 4^{\alpha+\beta} \int_{B_R(0) \setminus \overline{B_\delta(0)}} \|\zeta\|^{\gamma-\alpha-\beta} dV_{\C^n}(\zeta)\\
&\leq& 4^{\alpha+\beta} C_{\alpha+\beta-\gamma}  \left\{ 
\begin{array}{ll}
R^{2n+\gamma-\alpha-\beta} & \ ,\ \alpha+\beta<2n+\gamma,\\
|\log R|+|\log \delta| & \ ,\ \alpha+\beta=2n+\gamma,\\
\delta^{2n+\gamma-\alpha-\beta} & \ ,\ \alpha+\beta>2n+\gamma,
\end{array}\right.
\end{eqnarray*}
The assertion follows easily from this in combination with \eqref{eq:es13}, \eqref{eq:es14} and \eqref{eq:es15}.
\end{proof}

For use in cut-off procedures, we need also:

\begin{lma}\label{lem:estimateCn4}
Let $n\geq 2$. Moreover,
let $0\leq \gamma\leq 6$ and $0 \leq \alpha,\beta <2n$ with $\alpha+\beta=2n+2\geq 6$. 
Then there exists a constant $C_3>0$ such that the following holds:
\begin{eqnarray*}\label{eq:estimateCn4}
\|x\|^{6-\gamma}\int_{B_{\epsilon_{k-1}}(0) \setminus \overline{B_{\epsilon_{k+1}}(0)}} \frac{\|\zeta\|^{\gamma -4}
dV_{\C^n}(\zeta)}{\|\zeta-x\|^\alpha \|\zeta+x\|^\beta \big|\log\|\zeta\|\big|} \leq  C_3
\end{eqnarray*}
for all $x \in \C^n$ and all $k\in\Z$, $k\geq 1$, where $\epsilon_k=e^{-e^k/2}$.
\end{lma}

\begin{proof}
Let $\delta:=\|x\|/3$ and set
$$D_1:= B_\delta(x)\ \ ,\ \ D_2:= B_\delta(-x).$$
Then $\|\zeta + x\|=\|(\zeta-x) + 2x\| \geq 5\delta$ and $\|\zeta\| \leq 4 \delta$ on $D_1$ and so we obtain
\begin{eqnarray*}
\int_{D_1} \frac{\|x\|^{6-\gamma} \|\zeta\|^{\gamma-4} dV_{\C^n}(\zeta)}{\|\zeta-x\|^\alpha \|\zeta+x\|^\beta \big|\log\|\zeta\|\big|}
&\lesssim& \frac{\delta^{6-\gamma+\gamma-4 -\beta}}{\log 4 + |\log\delta|} \int_{B_\delta(x)} \frac{dV_{\C^n}(\zeta)}{\|\zeta-x\|^\alpha}\\
&\leq& \frac{\delta^{2-\beta} C_\alpha \delta^{2n-\alpha}}{\log 4 + |\log\delta|} \lesssim 1
\end{eqnarray*}
by use of Lemma \ref{lem:estimateCn1} and $\alpha+\beta=2n+2$ (on the domain of integration, the $\log$-term only helps).
The integral over $D_2$ is treated completely analogous.

Finally, we have to consider 
$D := \big( B_{\epsilon_{k-1}}(0) \setminus \overline{B_{\epsilon_{k+1}}(0)}\big) \setminus \big(D_1\cup D_2\big)$.
Here, we can use $\|\zeta-x\|\geq \delta=\|x\|/3$ and $\|\zeta + x\|\geq \delta=\|x\|/3$ to eliminate $\|x\|$ in the numerator.
Moreover, we have
$$\|\zeta\| \leq \|\zeta-x\| + \|x\| = \|\zeta-x\| + 3\delta \leq 4 \|\zeta-x\|,$$
and analogously $\|\zeta\| \leq 4 \|\zeta+x\|$.
From that we deduce:
\begin{eqnarray*}
 && \int_{D} \frac{\|x\|^{6-\gamma}\|\zeta\|^{\gamma -4}
 dV_{\C^n}(\zeta)}{\|\zeta-x\|^\alpha \|\zeta+x\|^\beta \big|\log\|\zeta\|\big|}
\lesssim \int_{B_{\epsilon_{k-1}}(0) \setminus \overline{B_{\epsilon_{k+1}}(0)}} \frac{dV_{\C^n}(\zeta)}{\|\zeta\|^{2n} \big|\log\|\zeta\|\big|}\\
&\sim& \int_{\epsilon_{k+1}}^{\epsilon_{k-1}} \frac{- dt}{t \log t} = - \log (-\log t) \big|^{\epsilon_{k-1}}_{\epsilon_{k+1}} = -(k-1) + (k+1) = 2.
\end{eqnarray*}
\end{proof}

\end{document}